\documentclass{amsart}
\usepackage{comment}
\usepackage{mathtools,comment}
\usepackage{mathtools}
\usepackage[foot]{amsaddr}
\usepackage{etex}
\usepackage[usenames,dvipsnames]{pstricks} 
\usepackage{epsfig}
\usepackage{graphicx,color}
\usepackage{geometry}
\usepackage[all]{xy}
\usepackage{amssymb}
\usepackage{amsmath}
\usepackage{cite}
\usepackage{fullpage}
\xyoption{poly}
\usepackage{url}
\usepackage{hyperref}
\numberwithin{equation}{subsection}
\numberwithin{figure}{subsection}
\usepackage{longtable}
\usepackage{verbatim}
\usepackage{tikz}
\usepackage{tikz-cd}
\usetikzlibrary{arrows.meta,bending}

\usetikzlibrary{decorations.pathmorphing}

\newtheorem{theorem}{Theorem}[section]
\newtheorem*{theorem*}{Theorem}
\newtheorem{lemma}[theorem]{Lemma}
\newtheorem{proposition}[theorem]{Proposition}
\newtheorem{corollary}[theorem]{Corollary}
\newtheorem*{corollary*}{Corollary}
\theoremstyle{definition}
\newtheorem{definition}[theorem]{Definition}

\newtheorem{notation}[theorem]{Notation}
\newtheorem{remark}[theorem]{Remark}

\newtheorem{example}[theorem]{Example}

\newtheorem*{question*}{Question}

\newtheorem*{steps*}{Answer/steps}

\newtheorem*{progress*}{Progress}
\newtheorem*{example*}{Example}

\newtheorem*{remark*}{Remark}
\newtheorem*{remarks*}{Remarks}
\newtheorem*{definition*}{Definition}
\newtheorem{conj}[theorem]{Conjecture}
\newtheorem*{conj*}{Conjecture}

\usepackage[all]{xy}

\usepackage{forest}
\usepackage{multirow,qtree}

\forestset{
	my tree/.style={
		for tree={
			parent anchor=north,
			child anchor=south,
			grow=90,
		},
		nice empty nodes up
	},
	nice empty nodes up/.style={
		for tree={calign=center},
		delay={where content={}{shape=coordinate,for parent={for children={anchor=south}}}{}}
	}
}

\renewcommand{\tilde}{\widetilde}

\newcommand{\QQ}{\mathbb{Q}}

\newcommand{\ZZ}{\mathbb{Z}}

\newcommand{\PP}{\mathbb{P}}

\newcommand{\cR}{\mathcal{R}}

\newcommand{\into}{\hookrightarrow}

\newcommand{\ga}{G^{\text{arith}}}

\newcommand\restr[1]{\raisebox{-.5ex}{$|$}_{#1}}

\DeclareMathOperator{\Res}{Res}

\DeclareMathOperator{\Aut}{Aut}

\DeclareMathOperator{\Par}{Par}

\DeclareMathOperator{\Gal}{Gal}
\DeclareMathOperator{\sgn}{sgn}

\DeclareMathOperator{\im}{im}
\DeclareMathOperator{\id}{id}
\DeclareMathOperator{\ord}{or d}

\DeclarePairedDelimiter\floor{\lfloor}{\rfloor}
\newcommand{\ilim}{\mathop{\varprojlim}\limits}

\begin{document}
	
\title{iterated monodromy group of a PCF Quadratic non-polynomial Map}

\author{\"{O}zlem Ejder}
\address{Bah\c{c}e\c{s}ehir University, Faculty of Enginering, Y{\i}ld{\i}z, Be\c{s}iktas\c{},34349}
\email{ozlemejderff@gmail.com}
\thanks{The first-named author was supported by part by T\"{u}bitak and Marie Sk{\l}odowska-Curie actions grant 120C071.}

\author{Yasemin Kara}
\address{Bogazici University, Faculty of Arts and Sciences, Mathematics Department, Bebek, Istanbul, 34342}
\email{yasemin.kara@boun.edu.tr}

\author{Ekin Ozman}
\address{Bogazici University, Faculty of Arts and Sciences, Mathematics Department, Bebek, Istanbul, 34342}
\email{ekin.ozman@boun.edu.tr}
\subjclass[2020]{11G32, 12F10, 37P05, 37P15}

\date{}


\keywords{Belyi map, arboreal Galois group, dynamical sequence.}

\begin{abstract}
	
We study the postcritically finite non-polynomial map $f(x)=\frac{1}{(x-1)^2}$ over a number field $k$ and prove various results about the geometric $G^{\text{geom}}(f)$ and arithmetic $G^{\text{arith}}(f)$ iterated monodromy groups of $f$.  We show that the elements of $G^{\text{geom}}(f)$ are the ones in $G^{\text{arith}}(f)$   that are fixing the roots of unity by assuming a conjecture on the size of $G^{\text{geom}}_n(f)$.  Furthermore, we describe exactly for which $a \in k$ the Arboreal Galois group $G_a(f)$ and $G^{\text{arith}}(f)$ are equal.

\end{abstract}

\date{}
\maketitle

\section{Introduction}
Let $f$ be a rational map of degree $d$ defined over a number field $k$. For certain rational maps, the iterates, $f^n$, of these maps have been heavily studied for both arithmetic and dynamical interests. For instance, explicit descriptions of  geometric and arithmetic iterated monodromy groups ($G^{\text{geom}}(f)$  and $G^{\text{arith}}(f)$   respectively) and its specializations ($G_a(f), a \in k$) have been given for certain post critically finite (PCF) polynomials, but no non-polynomials have been analyzed until now.

These special groups naturally appear as subgroups of the automorphism group of a certain binary rooted tree $T$, $\Aut(T)$, and it is a natural question to describe their actions on each finite level of $T$. These problems hold great importance in arithmetic applications and have captured considerable attention. A rapidly expanding body of literature is dedicated to investigating specific cases. However, existing work on these questions focuses largely on the case of PCF polynomials: \cite{Pinkpolyn}, \cite{arithmeticbasilica}, \cite{BostonJones07}, \cite{BostonJones09},\cite{FP20}, \cite{BEK}, \cite{BelyiEjder}.  The present article is the first one to consider the case of a PCF non-polynomial quadratic rational function, namely $f(x)=1/(x-1)^2$.

In this article, we investigate the computation of geometric and arithmetic iterated monodromy groups associated with $f(x)=1/(x-1)^2$.  Subsequently, our investigation explores the  interplay between a family of Arboreal Galois groups (emerging through the specialization of the basepoint to an element of a number field) and their integration within the arithmetic  iterated monodromy groups.

The main theorem of this article describes exactly for which $a \in k$ the Arboreal Galois group $G_a(f)$ and $G^{\text{arith}}(f)$ are equal for the PCF non-polynomial map $f(x)=1/(x-1)^2$.

\begin{theorem}\label{intromain}
	Let $a \in k$ and $K_{\infty,a}:=\bigcup_{n} K_{n,a}$ where $K_{n,a}$ is the splitting field of $f^{n}(x)-a$ over $k$ for $n \geq 1$.  Let $G_a(f)$ be the inverse limit of $G_{n,a}(f):=\Gal(K_{n,a}/k)$. The following are equivalent:
	\begin{enumerate}
		
		\item The degree $| k(\sqrt{a},\sqrt{a-1},\zeta_8):k|=16$,  \label{1}
		\item $G_a(f)=G^{\text{arith}}(f)$,\label{2}
		
		\item $G_{5,a}(f)=G^{\text{arith}}_5(f)$. \label{3}
	\end{enumerate}
	
\end{theorem}

We note that while $f$ is a dynamical Belyi map, it is not normalized. The iterates of $f$ are similarly dynamical Belyi maps but they do not have a unique ramification point above each branch point $0,1,\infty$, hence the results of \cite{BEK} do not apply to $f$ or any iterates of $f$. In fact, no quadratic maps are covered by the results of \cite{BEK}.

When working with PCF maps in general, there are various technical challenges due to the fact that the  iterated monodromy groups are small inside $\Aut(T)$.  One can roughly divide these difficulties in three parts:
\begin{enumerate}
	\item \label{prob:conj}The conjugacy problem for generators of the geometric iterated monodromy groups: PCF nature of the map implies the geometric  iterated monodromy groups possess a finite set of generators. These generators, up to conjugacy in $\Aut(T)$, exhibit a simple recursively-defined action on $T$. However, each generator may require conjugation by a different element of $\Aut(T)$. This is a major obstacle, and one hopes that a simultaneous conjugation will suffice. Without this there is no obvious path towards understanding the geometric  iterated monodromy groups.
	
	\item \label{prob:constants} Calculating the constant field subextension: This  is analogous to finding the quotient of arithmetic iterated monodromy groups by geometric iterated monodromy groups. This often involves arithmetic questions such as locating roots of unity in the relevant subextensions and calculating discriminants.  The overall aim is to demonstrate that the quotient is manageable which usually requires showing that the geometric iterated monodromy group is large. See also Diagram~\ref{eq:exact}
	\item \label{prob:special} Analysis of specializations: Often this takes the form of a sufficient condition for a specialization to give the entire arithmetic  iterated monodromy groups. In general there is a lot of interest in showing which specializations give a finite-index subgroup of the arithmetic  iterated monodromy groups. In conjecture, this is expected to happen for almost all instances, except for a finite number of specific specializations. Naturally, this is a much harder problem.
	
	\end{enumerate}

 In this article we work with the rational map $f(x)=1/(x^2 - 1)$, and we completely resolve Problem~\ref{prob:conj} in the list above, making use of the fact that the postcritical set is remarkably small. Progress has also been achieved concerning Problem~\ref{prob:constants}, with the primary challenge lying in determining the sizes of the natural quotients of the geometric iterated monodromy groups.

 Finally, we tackle  Problem~\ref{prob:special} by giving a sufficient condition for a specialization to give the entire arithmetic  iterated monodromy groups, which is similar to the condition given in the work of \cite{arithmeticbasilica}. 
 
 Furthermore, we also think that the map $f(x)=1/(x^2 - 1)$ can provide a negative answer to a question raised by Jones and Levy (\cite[Question~2.4]{Jones-Levy}). This will be investigated in future work.

In the rest of the introduction, we give the main notions and a brief survey of the literature, outline the content of the paper and state our results other than the main theorem which was given as Theorem \ref{intromain} at the beginning of the section.

 \subsection{Definitions, Literature, and Results}\label{sec:def}
 For each $n\geq 1$, define the $n$-th iterate of $f$ by $f^n=f\circ \ldots\circ f$. Let $t$ be a point in $\PP_{\bar{k}}^1$. If $f^n(x)-t$ is separable, then there are $d^n$ points in the set $\{ x \in \PP_{\bar{k}}^1 | f^n(x)=t\}$.   
Let $P$ be the postcritical set of $f$, that is, 
$$P=\{f^n(x) \in \PP^1_k : x \text{ is a critical point of } f \text{ and } n\geq 1\}.$$  A regular $d$-ary tree $T$ associated to $f$ can be constructed  as follows: 
\begin{itemize}
	\item the leaves of $T$ are the points in the preimage set  $f^{-n}(x_0)$ for $n\geq 1$ where $x_0\in \PP^1_k\setminus P $, 
	\item two leaves $v,w$ are directly connected if $f(v)=w$.
\end{itemize}

The iterates $f^n$ are unbranched outside the set $P$.  This gives a representation of the \'etale fundamental group of $\PP^1_k\backslash P$ inside $\Aut(T)$.  The image of this representation is called \emph{the arithmetic iterated monodromy group} of $f$ and denoted by $G^{\text{arith}}(f)$. Similarly, a representation of the  \'etale fundamental group of $\PP^1_{\bar{k}}\backslash P$ inside $\Aut(T)$ is called the  \emph{the geometric iterated monodromy group} of $f$ and denoted by $G^{\text{geom}}(f)$.

Equivalently we can construct these groups as follows: let $t$ be a transcendental element over $k$ and $k(t)$ be the field of rational functions. Let $K_n$ be the splitting field of $f^{n}(x)-t$ over $k(t)$ for $n \geq 1$. Then the Galois groups $G_n=\Gal(K_n/k(t))$ form an inverse system and the limit of this system is the arithmetic iterated monodromy group. A similar setup where $k(t)$ is replaced by $\bar{k}(t)$ results in the geometric iterated monodromy group. Finally, we can construct Galois groups $G_{n,a}(f)$ by looking at the solutions of $f^n(x)=a$ for $a\in k$ and $G_{a}(f)$ by taking limit over $n$. The Galois groups obtained in the latter construction are called Arboreal Galois groups and the study of such groups started with the work of Odoni in the 1980s, see \cite{Odoni85,Odoni2} and \cite{Odoni88}. We view $G^{\text{geom}}(f)$ naturally as a normal subgroup of $G^{\text{arith}}(f)$. One may see the Arboreal Galois group $G_a(f)$ as the specialization of $G^{\text{arith}}(f)$ at $t=a$. As it is shown in \cite{Odoni85}, Galois groups do not grow
under specializations. In a loose sense, $G^{\text{arith}}(f)$ is the group $G_a(f)$ for a generic choice of $a$.

A map $f$ is called \emph{postcritically finite (PCF)} if the orbit of each critical point is finite. It is known that, for PCF maps, the geometric iterated monodromy group is topologically finitely generated. Characterizing for which maps $G_a$ embeds as a finite index subgroup inside $\Aut(T)$ is a main question in the field and this question can be viewed as an analogue of Serre’s open image theorem.  On the other hand, from the works of Jones\cite{Jonessurvey} and Pink\cite{Pinkpolyn}, $G_a$ has infinite index inside  $\Aut(T)$ for PCF maps.

The case where the degree of $f$ is $2$, i.e., the quadratic case was extensively studied. 
 In \cite{Pinkpolyn} and \cite{Pinkrational}, Pink shows that the arithmetic and geometric iterated monodromy groups of the quadratic PCF polynomials and quadratic morphisms with infinite postcritical set $P$ are determined only by the combinatorial data of the postcritical set $P$. Moreover, he answers Problem~\ref{prob:conj} and Problem~\ref{prob:constants} for such maps completely. Pink's work in particular shows for $g(x)=x^2-1$ (Basilica map) that the field $K_{\infty}(g)=\cup K_n(f)$ contains a $2^j$'th root of unity for any $j\geq 1$. In \cite{arithmeticbasilica}, the authors give an explicit construction of these roots of unity and an explicit description of geometric and arithmetic iterated monodromy groups. 

This article studies the map $f(x)=\frac{1}{(x-1)^2}$. The critical points of the map $f$ are $1$ and $\infty$. The postcritical set of the map $f$ consists of the points $\{0,1,\infty\}$ and it is periodic:
\[ 0 \to 1 \to \infty \to 0 \]
Similar to \cite{arithmeticbasilica}, we construct a $2^n$th root of unity for each $n\geq 1$ in $K_{\infty}$. This result also follows from a more general result of Hamblen and Jones \cite[Corollary~2.4]{hj}. In addition, we explicitly describe the finite levels where each roots of unity occurs. See Theorem~\ref{unity}(2).


Based on our calculations, using MAGMA\cite{magma}  and the group theoretical arguments given in Section~\ref{furtherG}, we conjecture that:

\textbf{Conjecture 7.15.}
	$\log_2 | G^{\text{geom}}_{n}(f)|=2\log_2|G_{n-1}^{\text{geom}}(f)| - \floor*{\frac{n-3}{2}}$.

Pink \cite[Proposition~2.3.1]{Pinkpolyn} calculates the sizes of the finite levels of the iterated geometric monodromy group in the case of polynomial PCF maps. 
He proves this result by studying certain normal subgroups $N$ defined for any generator $a$. For the PCF polynomials with periodic postcritical set, $G^{\text{geom}}$ is isomorphic to the semidirect product $N \rtimes \langle a\rangle$. In the case of polynomials with strictly preperiodic postcritical set, there is no semidirect product but the intersection of $N$ and $\langle a\rangle$ is finite. For the non-polynomial map $f(x)=\frac{1}{(x-1)^2}$, the size of the intersection of $N$ and $\langle a \rangle$ is infinite. Because of this obstruction, it is not possible to use Pink's approach to calculate the size of $G^{\text{geom}}$. 

It may also be worth noting that in the polynomial case (see \cite[Section~2.7 and 3.9]{Pinkpolyn}), there is an abundance of odometers in $G^{\text{geom}}$, however for the map considered in this article, there are no odometers, at all. (An easy calculation shows there is no element of order $8$ at level $3$.)
The lack of odometers may also explain the difficulty in understanding this group.

Based on Conjecture \ref{conj:G} , we have the following result (Theorem~\ref{thm:cyc}).
\begin{theorem}
	Assuming Conjecture \ref{conj:G}, the quotient group $\ga(f)/G^{\text{geom}}(f)$ is isomorphic to $\ZZ_2^*$ and the homomorphism $\Gal(\bar{k}/k) \to \ZZ_2^*$ is given by the cyclotomic character.
\end{theorem}

We note that in \cite{LMY}, it is shown that there are only twelve quadratic rational PCF maps over $\mathbb{Q}$ up to conjugacy. Recent work of Ferraguti-Ostafe-Zannier \cite{fos} shows that for the rational maps $f$ that are $\mathbb{Q}$-conjugate to $\frac{2}{(x-1)^2}$ or $\frac{2x^2+4x+4}{-x^2+2}$, the group $\ga(f)/G^{\text{geom}}(f)$ has a non-abelian quotient, hence the statement of Theorem~\ref{thm:cyc} doesn't hold for these maps. See \cite[proof of Theorem~3.8]{fos} for details. These two maps both have a postcritical set that is strictly preperiodic. It would be interesting to find out if the above theorem holds for any other non-polynomial map in this list. 

For a closed subgroup $G$ of  $\Aut(T)$, the size is measured by its Hausdorff dimension which is defined as $\lim_{n\to \infty} \frac{\log_2|G_n|}{\log_2|W_n|}$(See Section \ref{sec:background} for definitions of these notations). If the conjecture holds, we calculate that the Hausdorff dimension of $G$ is $2/3$.

\subsection{Outline}

In Section~\ref{sec:rigid}, we describe the topological generators of $G^{\text{geom}}(f)$ recursively and show that $f$ satisfies a semi-rigidity property (Proposition~\ref{prop:conj}). Let $K_{\infty}=\bigcup_{n} K_n$ and $K_{n}$ is the splitting field of $f^{n}(x)-t$ over $K=k(t)$ for $n \geq 1$. In Section~\ref{sec:roots}, we prove that the field $K_{\infty}$ contains a primitive $2^m$th root of unity for all $m \geq 1$.
Section~\ref{sec:main} is reserved for the proof of Theorem~\ref{intromain}. One important ingredient in the proof is the discriminant calculation for the iterates of $f$ and this is done in section~\ref{sec:disc}. 
In Section~\ref{furtherG}, we study the group $G^{\text{geom}}(f)$ in more detail. We find an upper bound on the size of $G^{\text{geom}}(f)$ which agrees with Conjecture~\ref{conj:G}.

\subsection*{Acknowledgements:} We extend our gratitude to the referees for their invaluable feedback and corrections, particularly in relation to the suggestions concerning the structure and presentation of the introduction.

\section{Background}\label{sec:background}

\subsection{Automorphisms of the $d$-ary regular tree}\label{sec:tree}

Let $T$ be the infinite regular $d$-ary tree whose vertices are the finite words over the alphabet $\{1,2\}$. For any integer $n\geq 1$, we let $T_n$ denote the finite rooted subtree whose vertices are the words of length at most $n$. We call the set of words of length $n$  \emph{the level $n$} of $T$. We use the notation $W:=\Aut(T)$ and $W_n:=\Aut(T_n)$. 

We label the vertices of the tree as follows: the root of the tree corresponds to the level $0$, and has the empty label $()$. The vertices at the level $i$ are labeled as 
$(\ell_1,\ldots,\ell_i)$ with $\ell_j \in \{1,2\}$. Here $(\ell_1,\ldots,\ell_{i-1})$ is the unique vertex at level $i-1$ which is connected to $(\ell_1,\ldots,\ell_i)$ by an edge. For the vertices of $T_n$ at level $n$,
also called the \emph{leaves}, we additionally use the numbering
\begin{equation}\label{eq:leafnumber}
1+\sum_{k=1}^n(\ell_k-1) 2^{k-1}\in \{1,2,3,\ldots, 2^n\} \qquad \text{ instead of } (\ell_1,
\ldots, \ell_n).
\end{equation} in the examples.

Since $W_n$ acts faithfully on the leaves of the tree $T_n$, the
choice of the numbering induces an injective group homomorphism
\begin{equation}\label{eq:iota}
\iota_n: W_n \hookrightarrow S_{2^n}.
\end{equation}

\begin{example}
	Here is an example of a regular $2$-ary rooted tree of level $3$.	
\begin{center}
\scalebox{0.7}{	
\begin{forest}
	my tree
	[$\bullet$
	[(2)
	[(22)
	[(222)]
	[(221)]
	]
	[(21)
	[(212)]
	[(211)
	]]]
	[(1)
	[(12)
	[(122)]
	[(121)]
	]
	[(11)
	[(112)]
	[(111)
	]]]
	]
\end{forest}}
\end{center}

\end{example}

We embed  $W \times W$ into $W$ by identifying the complete subtrees rooted at level one of the tree $T$ with $T$ itself. The image of the embedding $W\times W \into W$ is given by 
the set of automorphisms acting trivially on the first level. 
The exact sequence $1 \to W\times W  \to W \to S_2 \to 1$ splits and gives the semi direct product 
\[ W \simeq (W\times W) \rtimes S_2 \text{ and }  W_n \simeq (W_{n-1} \times W_{n-1}) \rtimes S_2. \]
In other words, $W$ and $W_n$ have a wreath product structure:
\begin{equation}\label{eq:wr}
W \simeq W \wr S_2 \text{ and } W_n \simeq W_{n-1} \wr S_2 \end{equation}
for $n\geq 2$. 

This isomorphism in \ref{eq:wr} is induced by the two complete subtrees of $T_n$ at level $1$ that are copies of $T_{n-1}$.  This allows us to write the elements of $W_n$ as $(u,v)\tau$ where $u,v \in W_{n-1}$ and $\tau \in S_2$. 

We have the following relation in $W$ arising from the wreath product:
\begin{equation}\label{eq:relations} (x_1, x_2)\tau (y_1, y_2)\tau'= (x_1y_{\tau(1)},x_2y_{\tau(2)})\tau\tau' 
\end{equation}

Let $\sigma=(12) \in S_2$. We denote the automorphism that permutes two subtrees at level $1$ of the tree $T_n$ also by $\sigma$. With the notation above,

\begin{equation}\label{eq:sigma}
	\sigma:=(\id,\id)\sigma \in W_n \text{  for } n \geq 2.\
\end{equation} where $\id$ denotes the identity automorphism in $W_{n-1}.$

For every $n\geq 1$, we write $\pi_n$ for the natural projection
\begin{equation}\label{not:pi}
\pi_{n}: W \to W_n,
\end{equation}
which corresponds to restricting the action of an element of
$W$ to the subtree $T_m$ consisting of the levels $0, 1, \ldots, m$.

Similarly, for any $m \geq n$, we denote the natural projection $W_m \to W_n$ by $\pi_{m,n}$. We abuse the notation whenever the domain is clear and write it as $\pi_n$.
We denote the image of an element $w \in W$ (or $W_m$ for $m\geq n$ ) under $\pi_n$ as $w\restr{T_n}$.

Let $H$ be a subgroup of $W$. For each $n\geq 1$, we define $H_n:=\pi_n(H) \subset W_n$.

\begin{example}\label{ex:aut2}
	The automorphism group of the first level tree, $W_1=\Aut(T_1)$ is $S_2$ consisting of the identity automorphism and the transposition of the two leaves. The automorphism group of the second level tree, $W_2=\Aut(T_2)$,  is given by $W_1 \wr S_2=(W_1 \times W_1) \rtimes S_2$ and is isomorphic to the dihedral group of $8$ elements. Using the labeling below we can embed $W_2$ into $S_4$ and write the elements of $W_2$ as follows:
	
	\[W_2= \{ id, (13), (24), (13)(2 4), (1234), (1432), (12)(34), (14)(32)  \} \]
	
	The natural projection $\pi_{2,1}: W_2 \rightarrow W_1$ sends the first four elements of $W_2$ listed above to the identity automorphism on the red $T_1$ and the last 4 elements to the transposition on the red $T_1$.
	
	  \begin{center}    
		\includegraphics{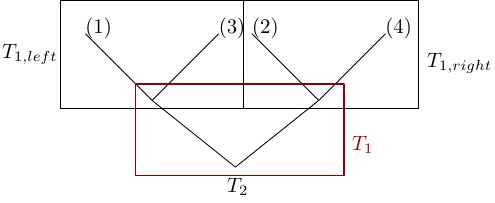}
	\end{center}

We can list the elements of $W_2$ in the wreath product notation as below.
$$id=(\id,\id)\id \quad (13)=(\sigma, \id)\id \quad (24)=(\id, \sigma)\id \quad (13)(24)=(\sigma, \sigma)\id $$ $$ (1234)=(\id,\sigma)\sigma \quad (1432)=(\sigma, \id)\sigma \quad (12)(34)=(\id,\id)\sigma \quad (14)(32)=(\sigma, \sigma ) \sigma$$

It is possible to see that the group $W_3=W_2 \wr S_2 $ consists of elements of the form $(f,g)\tau$ where $f,g \in W_2$ and $\tau \in S_2$. More precisely an arbitrary element of $W_3$ is of the form $((f_1,f_2)\tau_f, (g_1, g_2)\tau_g )\tau$ where each component is an element of the symmetric group $S_2$. Therefore $W_3$ has size $128$.
\end{example}


\subsection{Iterated monodromy groups}\label{sec:mon}
Let $k$ be a number field, fix an algebraic closure $\bar{k}$ of $k$.  Let $f:\PP^1_k \to \PP^1_k$ be a morphism of degree $d$ defined over $k$. Let $C$ be the set of critical points of $f$ and let $P$
be the forward orbit of the points in $C$, i.e.,
\[ P:=\{ f^n(c) : n\geq 1, c \in C\}. \]

Let $x_0 \in \PP^1_k(k)\backslash P$. Then each $f^n$ is a connected unramified covering of $\PP^1_k \backslash P$, hence it is determined  by the monodromy action of $\pi_1^{\acute{e}t}(\PP_k^1 \backslash P, x_0)$ on $f^{-n}(x_0)$ up to isomorphism. Let $T_{x_0}$ be the tree defined as follows: it is rooted at $x_0$, the leaves of $T_{x_0}$ are the points of $f^{-n}(x_0)$ for all $n\geq 1$, and the two leaves $p,q$ are connected if $f(p)=q$.  Varying $n$, associated monodromy defines a representation 
\begin{equation}\label{eq:rho}
 \rho: \pi_1^{\acute{e}t}(\PP_k^1 \backslash P, x_0) \to \Aut(T_{x_0})
\end{equation} 
whose image we call the \emph{arithmetic iterated monodromy group} $G^{\text{arith}}(f)$ of $f$. One can also study this representation over $\bar{k}$ and obtain 
\[ \pi_1^{\acute{e}t}(\PP_{\bar{k}}^1 \backslash P, x_0) \to \Aut(T_{x_0}).
\]
 We call the image of the map in this case the \emph{geometric iterated monodromy group} $G^{}(f)$. We note that these iterated monodromy groups are unique up to conjugation by the elements of $\Aut(T_{x_0})$. The arithmetic and the geometric iterated monodromy groups of $f$ fit into an exact sequence as follows:
 \begin{equation}\label{eq:exact}
   \begin{tikzcd}
1\arrow{r}  &  \pi_1^{\acute{e}t}(\PP_{\bar{k}}^1 \backslash P, x_0) \arrow{r} \arrow{d}& \pi_1^{\acute{e}t}(\PP_k^1 \backslash P, x_0) \arrow{r} \arrow{d} &\Gal(\bar{k}/ k) \arrow{r} \arrow{d} &1  \\
   1 \arrow{r} &   G^{\text{geom}}(f) \arrow{r} &G^{\text{arith}}(f) \arrow{r}  &\Gal(F /k) \arrow{r} &1 \\
     \end{tikzcd}
  \end{equation}
for some field extension $F$ of $k$. This field $F$ is called the constant field subextension and determining this field $F$ or its degree is a fundamental problem as explained in item~\ref{prob:constants} given in the introduction.

Equivalently, we can define the geometric (arithmetic resp.) iterated monodromy group of $f$ as the projective limit of the Galois group of the splitting field of the $n$th iteration $f^{n}(x)-t$ over $\overline{k}(t)$ (over $k(t)$ resp.). By abuse of notation we denote the geometric and the arithmetic iterated monodromy groups of $f$ as follows:
\[ G^{\text{geom}}(f):=\ilim_n \Gal((f^{n}(x)-t)/\overline{k}(t)) \text{ and }  G^{\text{arith}}(f):=\ilim_n \Gal((f^{n}(x)-t)/k(t)). \]

The groups $G^{\text{geom}} (f)$ and $G^{\text{arith}}(f)$ are profinite groups and they are embedded into $\Aut(T_{x_0})$ by construction. We identify the tree $T_{x_0}$ we constructed above with the infinite $2$-ary tree $T$ we described in the previous section. From now on, we  assume $G^{\text{geom}}(f)$ and $G^{\text{arith}}(f)$ are subgroups of $W=\Aut(T)$. Both of these groups are self-similar, i.e. for any $n\geq 2$, we have 
 \[ G_n^{\text{arith}} \subset G_{n-1}^{\text{arith}} \times G_{n-1}^{\text{arith}} \text{ and } G_n^{\text{geom}} \subset G_{n-1}^{\text{geom}} \times G_{n-1}^{\text{geom}}.
 \]
 Also note that naturally $G^{\text{geom}}(f)$ is a subgroup of $G^{\text{arith}}(f)$. We  drop $f$ from the notation whenever it is clear and write $G^{\text{geom}}$ and $G^{\text{arith}}$.

\section{Geometric iterated monodromy group of $f(x)=\frac{1}{(x-1)^2}$}\label{sec:rigid}

The group $G^{\text{geom}}$ is a quotient of the fundamental group of $\PP_{\bar{k}}^1\backslash P$ and it is topologically generated by the elements $b_p$ for all points $p$ in the postcritical set $P$(defined in Section \ref{sec:def}). When $P$ is finite, the geometric iterated monodromy group, $G^{\text{geom}}$, is topologically finitely generated and the product of the generators is identity since this is a property that holds for the fundamental group. These generators, $b_p$s, are conjugate to certain elements under $W$ as proved in Proposition 1.7.15 in \cite{Pinkpolyn}.

From now on, we  fix $f(x)$ as $\frac{1}{(x-1)^2}$. Note that the set of critical points of $f(x)$ is $C=\{1, \infty\}$ and the postcritical set of $f(x)$ is $P=\{0,1,\infty\}$. Now we state the aformentioned result of Pink for this case:

\begin{proposition}\cite[Proposition 1.7.15]{Pinkpolyn}\label{prop:Pink}
	For any $p \in P$, the element $b_p$ is conjugate under $W$ to 
	
	\begin{itemize}
		\item $(b_c,1)\sigma$ if $p=f(c)$ where $c$ is an element in $C \cap P=\{1, \infty\}$
		\item $(b_0,1)$ if $p=f(0)=1.$
	\end{itemize}
	
\end{proposition}
Let $a_1, a_2, a_3$ be elements of $W=\Aut (T_{\infty})$ defined as follows:
\[ a_1=(\id,a_3), a_2=(\id,a_1)\sigma, a_3=(a_2,\id)\sigma \] where $a_2$ and $a_3$ act by $\sigma$ and $a_1$ acts as the identity on the first level of the tree; and here $\sigma$ is the non-identity element in $S_2$. Let $G$ be the subgroup of $W$ with topological generators $a_1,a_2,a_3$, i.e., $G=\langle a_1,a_2,a_3\rangle$. 

For any $n\geq 1$, let $G_n$ denote the image of $G$ under the natural projection $\pi_n: W \to W_n$. For any $n\geq m$, by restricting the projection maps $\pi_{m}$ to $G_n$ we get projection maps  between $G_i$'s $(\pi_{m}){\restr{G_n}}: G_n \rightarrow G_m$ for $n \geq m$. We drop $G_n$ from the notation where the domain is clear. If $b_i \in W_n$ is conjugate to $a_i\restr{T_n}$ under $W_n$, we  denote this by $b_i \sim a_i$. 

We first show that the product of the generators $a_1,a_2,a_3$ is the identity.
\begin{lemma} \label{a-rel}
	$a_1a_2a_3=\id$ for all $n\geq 1$.
\end{lemma}	
\begin{proof}
	When $n=1$,  $a_1a_2a_3=\sigma^2=\id$ as $a_2$ and $a_3$ act by $\sigma$ and $a_1$ acts as the identity on $T_1$. Now assume $a_1 a_2 a_3\restr{T_{n-1}}=\id$, then $(a_1 a_2)\restr{T_{n-1}}=a_3^{-1}\restr{T_{n-1}}$ and hence 
	\[a_3\restr{T_{n-1}}(a_1 a_2 )\restr{T_{n-1}}=a_3 a_1 a_2 \restr{T_{n-1}}=\id.\] 
We also calculate that
\begin{align*}
a_1a_2a_3&=(\id,a_3)(\id,a_1)\sigma(a_2,\id)\sigma \\ 
          &=(\id,a_3)(\id,a_1)(\id,a_2)\sigma^2 \\
          &=(\id,a_3a_1a_2)
\end{align*}
Hence \[a_1 a_2 a_3 \restr{T_n}=(\id,a_3 a_1 a_2\restr{T_{n-1}})=(\id,\id)=\id\] by induction and $a_1 a_2 a_3=\id$ in $W$ since its projection in $W_n$ is identity for all $n$. 
\end{proof}

\begin{lemma}\label{lem:conja}
 The elements $a_1,a_2,a_3$ are conjugate to $b_1,b_\infty,b_0$ respectively.
\end{lemma}
\begin{proof} By Proposition~\ref{prop:Pink}, $b_1=(b_0,\id), b_{\infty}=(b_1,\id)\sigma$ and $b_0=(b_{\infty},\id)\sigma$. The statement follows from \cite[Lemma 1.3.1]{Pinkpolyn}.
\end{proof}
Next proposition defines the geometric iterated monodromy group up to conjugacy. Our proof is an adjustment of the proof of \cite[Proposition~2.4.1]{Pinkpolyn} with the additional assumption that the product of the generators is identity. 
\begin{proposition} \label{prop:conj}
	For any elements $b_i \in W_n$ such that $b_i \sim a_i\restr{T_n}$ and satisfying  $b_1b_2b_3=\id$, there exists $w \in W_n$ and $x_i \in G_n$ such that $b_i=(wx_i)(a_i\restr{T_n})(wx_i)^{-1}$ for each $i$.  We denote this statement as $*_n$.
\end{proposition}
\begin{proof} By assumption $b_1 \sim a_1\restr{T_n}=(\id , a_3\restr{T_{n-1}}).$ By \cite[Lemma~1.2.1]{Pinkpolyn} we have $b_1=( c_3 ,\id  )$ or $( \id, c_3 )$ for some $c_3 \in W_{n-1}$ that is conjugate to $a_3\restr{T_{n-1}}$ .
	
Assume $b_1=(c_3,\id)$ where $c_3 \sim a_3\restr{T_{n-1}} \in W_{n-1}$. Let $b_1'=(\id,c_3)$ and assume there is $w \in W_n$ and $x_1',x_2,x_3 \in G_n$ such that the statement holds for $b_1',b_2,b_3$. Then taking, $x_1=x_2a_2x_2^{-1}x_1'$ (which is in $G_n$), we see that $w, x_1,x_2,x_3$ satisfy the statement for $b_1,b_2,b_3$. This reduces the proof to the case $b_1=( \id, c_3 )$. 
	
	We assume $b_1=(\id,c_3)$ where $c_3$ is conjugate to $a_3\restr{T_{n-1}} $. Since $b_2$ is conjugate to $a_2\restr{T_{n}}$, there is some $f \in W_n$ satisfying $b_2=fa_2f^{-1}$. If $f\restr{T_1} \neq \id$ then we can replace $f$ with $f'=f a_2$. Note that $f'\restr{T_1}= \id$ and we still get $b_2 \sim a_2$ through $f'$. Therefore we may assume that $f=(d,e)$ and therefore \begin{equation}\label{b2}
		b_2=fa_2f^{-1}=(d,e)(\id, a_1)\sigma(d^{-1}, e^{-1})=(de^{-1}, ea_1d^{-1}) \sigma.
		\end{equation} 
	
Using a similar argument we can deduce that $b_3=(g,h)(a_2, \id)\sigma(g^{-1}, h^{-1})=(ga_2h^{-1}, hg^{-1}) \sigma$ for some $(g,h) \in W_n$. 

Let's set $c_1:=ea_1e^{-1}$. We note that the element $c_3$ is not related to $c_1$. We will see shortly that $\{c_1,(c_1c_3)^{-1},c_3\}$ satisfies the property $*_{n-1}$.
	
	By assumption $b_1b_2=b_3^{-1}$ and is conjugate to $a_1a_2=a_3^{-1}$ under $W_n$. Therefore:
	
	\begin{align*}
		b_1b_2& =  (\id, c_3) (de^{-1}, ea_1d^{-1}) \sigma \\
		             & =(\id, c_3) (de^{-1}, ea_1e^{-1}ed^{-1}) \sigma\\
		             & =(\id, c_3) (de^{-1}, c_1(de^{-1})^{-1}) \sigma \\
		             & = (de^{-1}, c_3c_1(de^{-1})^{-1})\sigma
	\end{align*}

is conjugate to $a_1a_2=(\id, a_3a_1)\sigma$ under $W_n$. Using Lemma \ref{lem:conj}, this implies that  $de^{-1}c_3c_1(de^{-1})^{-1} \sim a_3a_1$ or $c_3c_1(de^{-1})^{-1}de^{-1} \sim a_3a_1.$ In either case, $c_3c_1 \sim a_3a_1$. 

Let's define $c_2:=(c_3c_1)^{-1}$. Now $\{c_1, c_2, c_3\}$ satisfies the statement $*_{n-1}$, hence by the induction assumption, there is  $u \in W_{n-1}$ and $x_i \in G_{n-1}$ such that 
\begin{equation}\label{ind} c_i=(ux_i) a_i\restr{T_{n-1}}(ux_i)^{-1}.
\end{equation}

Now we define $w \in W_n$ such that $b_2=wa_2\restr{T_n}w^{-1}.$ Let $w=(w_0,w_1)$ where $w_0=de^{-1}w_1$ and $w_1=ux_1$ . Here $u,x_1,d,e$ are as defined above. Then we get the following:

\begin{align*}
	wa_2\restr{T_n}w^{-1}&=(w_0,w_1)(\id, a_1)\sigma(w_0^{-1},w_1^{-1})\\
	&=(w_0w_1^{-1},w_1a_1w_0^{-1})\sigma\\
	&=(de^{-1},c_1ed^{-1})\sigma
	\end{align*}
The last equality follows from Equation (\ref{ind}). Now recall that by (\ref{b2})
$b_2=(de^{-1}, ea_1d^{-1}) \sigma$ and $c_1=ea_1e^{-1}$. Hence we get 
 
 \begin{align*}
 	b_2&=(de^{-1}, ea_1d^{-1}) \sigma \\
 	&=(de^{-1}, c_1ed^{-1}) \sigma\\
 	&=wa_2\restr{T_n}w^{-1}.
 		\end{align*}
 	
 	Now let's see that the same $w$ as above works for $b_1$ and $b_3$. We start with $b_1$.
 	
 	 \begin{align*}
 	 	w^{-1}b_1w&=(w_0^{-1}, w_1^{-1})(\id, c_3)(w_0,w_1)\\
 	 	&=(w_0^{-1}w_0, w_1^{-1}c_3w_1)\\
 	 	&=(\id, w_1^{-1}(ux_3)a_3(ux_3)^{-1}w_1)\\
 	 	&=(\id, (ux_1)^{-1}(ux_3)a_3(ux_3)^{-1}(ux_1))\\
 	 	 	&=(\id, (x_1^{-1}x_3)a_3(x_1^{-1}x_3)^{-1})\\
 	 	 	&=(x_1^{-1}x_3,x_1^{-1}x_3 )(\id,a_3)(x_1^{-1}x_3,x_1^{-1}x_3)^{-1}\\
 	 	\end{align*}	
 	 	Note that $x_1^{-1}x_3 \in G_{n-1}$. Hence $y_1=(x_1^{-1}x_3,x_1^{-1}x_3) \in G_n$ and we get
 	 	
 	 	$$w^{-1}b_1w=y_1a_1y_1^{-1}.$$
 	 	
 	 	Now we check $w^{-1}b_3w.$
 	 		\begin{align*}
 	 		w^{-1}b_3w&=(w_0^{-1}, w_1^{-1})(ga_2h^{-1}, hg^{-1})\sigma (w_0,w_1)\\
 	 		&=(w_0^{-1}ga_2h^{-1}w_1, w_1^{-1}hg^{-1}w_0) \sigma \\
 	 		&=((de^{-1}w_1)^{-1}ga_2h^{-1}w_1, w_1^{-1}hg^{-1}de^{-1}w_1)\sigma\\
 	 		&=(w_1^{-1}ed^{-1}ga_2h^{-1}w_1, w_1^{-1}hg^{-1}de^{-1}w_1)\sigma \\
  			\end{align*}
 	 	Now using the relation $b_1b_2=b_3^{-1}$ we get $(\id, c_3)(de^{-1},ea_1d^{-1})\sigma=(\id, c_3)(de^{-1},c_1ed^{-1})\sigma=(de^{-1},c_3c_1ed^{-1})\sigma=(hg^{-1},ga_2h^{-1})^{-1}\sigma$. Hence we have the following equalities:
 	 	
 	 	$$de^{-1}=gh^{-1}, c_3c_1d^{-1}=ha_2^{-1}g^{-1}$$
 	 	
 	 	Since $c_3c_1=c_2^{-1}$ we get $c_2=ed^{-1}ga_2h^{-1}=ha_2h^{-1}=ux_2a_2(ux_2)^{-1}$, therefore $x_2^{-1}u^{-1}ha_2(x_2^{-1}u^{-1}h)^{-1}=a_2$.
 	 	
 	 	Now we get using $w_1=ux_1$:
 	 	
 	 	\begin{align*}
 	 		w^{-1}b_3w&=(w_1^{-1}ha_2h^{-1}w_1, \id)\sigma \\
 	 		&=(x_1^{-1}u^{-1}ha_2h^{-1}ux_1, \id) \sigma \\
 	 		&=(x_1^{-1}x_2a_2x_2^{-1}x_1, \id)\sigma\\
 	 		&=(x_1^{-1}x_2,x_1^{-1}x_2 )(a_2, \id)\sigma (x_1^{-1}x_2, x_1^{-1}x_2)^{-1}\\
 	 		&=y_3a_3y_3^{-1}\\
 	 	\end{align*}
 	 	 
 	 	where $y_3=(x_1^{-1}x_2,x_1^{-1}x_2 ) \in G_n$
	\end{proof}

\begin{lemma}\label{lem:conj}
	
	if $(u,v)\sigma \sim (\id,w) \sigma$ then $uv \sim w.$
\end{lemma}
\begin{proof}
If $w:=(u,v)\sigma$ is conjugate to $(\id,w)\sigma$ in $W$, then $w^2=(uv,vu)$ is conjugate to $(w,w)$ in $W$. 
\end{proof}

By Proposition~\ref{prop:conj}, the group $G_n^{\text{geom}}$ is conjugate to $G_n$ for all $n\geq 1$. Taking limit over $n$, we have that $G^{\text{geom}}$ is conjugate to $G$ which allows us to identify $G^{\text{geom}}$ with $G$. 

\begin{notation}
	For the rest of the paper, we  assume that $G_n^{\text{geom}}$ equals $G_n$ for all $n\geq 1$ and $G^{\text{geom}}$ equals $G$.
\end{notation}
\section{The roots of unity}\label{sec:roots}

Let $K$ denote the field $k(t)$ and let $K_{n}$ be the splitting field of $f^{n}(x)-t$ over $K$ for $n \geq 1$. Define $K_{\infty}:=\bigcup_{n} K_n$. Then $G^{\text{arith}}=\Gal(K_{\infty}/K)$. We prove in Theorem~\ref{unity} that the field $K_{\infty}$ contains a primitive $2^m$th root of unity for all $m$, in particular the field $F$ in equation~\ref{eq:exact} is an infinite extension of $k$. 
  \begin{lemma} \label{lemma1} Let $\alpha$ denote a vertex of the infinite tree $T_{\infty}$ and let $f^{-1}(\alpha) = \{\alpha_1, \alpha_2\}$. Similarly let $f^{-1}(\alpha_i)=\{\alpha_{ij}: j=1,2\}$. We have
	\begin{enumerate}
	\item  $\alpha_1 \alpha_2=1-\frac{1}{\alpha}=\frac{\alpha-1}{\alpha}$.
	\item  $[(\alpha_{11}-1)(\alpha_{21}-1)]^2=\frac{1}{\alpha_1\alpha_2}=\frac{\alpha}{\alpha-1}$. 
	\end{enumerate}
	\end{lemma}
	\begin{proof} 
	Since $\frac{1}{(x-1)^2}=\alpha$ we get the roots $\alpha_1,\alpha_2$ satisfy $\alpha_i=1 \pm \frac{1}{\sqrt{\alpha}}$ hence the first result follows. Using the notation $\alpha_{11}=1+\frac{1}{\sqrt{\alpha_1}}, \alpha_{12}=1-\frac{1}{\sqrt{\alpha_1}}, \alpha_{21}=1+\frac{1}{\sqrt{\alpha_1}}, \alpha_{22}=1-\frac{1}{\sqrt{\alpha_1}}$ and by the first part we get $[(\alpha_{11}-1)(\alpha_{21}-1)]^2=\frac{1}{\alpha_1\alpha_2}=\frac{\alpha-1}{\alpha}.$
	\end{proof}

	Let $\alpha$ be any vertex in the tree $T_\infty$. 
	We  denote the two vertices in $f^{-1}(\alpha)$ by $\alpha_{1}$ and $\alpha_{2}$ where $\alpha_{1}=1+\frac{1}{\sqrt{\alpha}}$ and $\alpha_{2}=1-\frac{1}{\sqrt{\alpha}}$. As illustrated in the figure below $\alpha_{1}$ denotes the vertex to the left and $\alpha_{2}$ denotes the one to the right. 
	\begin{center} 
	\scalebox{0.7}{	
		\begin{forest}
			my tree
			[$\alpha$
			[$\alpha_2$
			[$\alpha_{22}$
			]
			[$\alpha_{21}$
		     ]]
			[$\alpha_1$
			[$\alpha_{12}$
			]
			[$\alpha_{11}$			
			]]
			]
	\end{forest}}
\end{center}
	Using this convention we get the following result:

	\begin{lemma}\label{lemma4e} For any vertex $\alpha$ of $T_{\infty}$, we have $\sqrt{\alpha} \in K_{\infty}$ and
		$$\alpha-1=\left[\frac{1}{(\alpha_{1}-1)(\alpha_{11}-1)(\alpha_{21}-1)} \right]^2$$
	\end{lemma}
	
	\begin{proof}
		The first claim follows from the fact that $\alpha_{i} -1 \in K_\infty$ and $\alpha=\left( \frac{1}{\alpha_{1}-1}\right)^2$.
	
	Similarly we have  $$f^{-1}(\alpha_{i})=\{\alpha_{i1}, \alpha_{i2}\} \text{ and } \alpha_{i}=\frac{1}{(\alpha_{i1}-1)^2}$$ for $i=1,2$. Combining these we see that the right hand side of the claimed equality is $\alpha \cdot \alpha_{1} \cdot \alpha_{2} $.  Using $\alpha_{1} \cdot \alpha_{2}=1-\frac{1}{\alpha}$
		we get that;
		$ \alpha \cdot \alpha_{1} \cdot \alpha_{2}=\alpha \cdot \left(1 -\frac{1}{\alpha} \right)=\alpha-1.$
	\end{proof}
	
We are now ready to prove the main theorem of this section. The first statement also follows from \cite[Corollary~2.4]{hj} which is stated for bicritical functions of any degree $d\geq 2$ whose critical points are both periodic. The second statement describes the explicit finite levels where each primitive root of unity is defined.  We note that whether these levels are minimum or not depends on the base field. For example when $K=\mathbb{Q}$, $K_n$ cannot contain $\zeta_{2^i}$th roots of unity for $i > \floor{\frac{n+1}{2}}$.
	
	\begin{theorem} \label{unity}
Let $m \geq 1$ be an integer.
\begin{enumerate}
 \item A primitive $2^m$th root of unity, $\zeta_{2^m} $ is in $K_\infty$ and is  equal to a product  of expressions  of the form $\beta_i-1$ and $(\beta_j-1)^{-1}$ for some vertices $\beta_i$ and $\beta_j$. 
 \item Moreover,  $K_n$ contains $\zeta_{2^m}$ for $n=2m-1$. 
 \end{enumerate}
\end{theorem}

	\begin{proof} 
	We first show that $K_\infty$ contains $\zeta_4$. This is done by showing that $-1$ is a square in $K_\infty$. In particular, we will see that $-1$ can be written as a square of a product of expressions of the form $\beta_i-1$ and $(\beta_j-1)^{-1}$ for some vertices $\beta_i$ and $\beta_j$. 
	
Consider \begin{equation}\label{m2}
			\left[\prod\limits_{i \in \{1,2\}} \frac{(t_{1i1}-1)}{(t_{2i1}-1)} \frac{t_{11}-1}{t_{21}-1}\right ]^2=\frac{(t_{111}-1)^2(t_{121}-1)^2}{(t_{211}-1)^2(t_{221}-1)^2} \;\; \frac{(t_{11}-1)^2}{(t_{21}-1)^2}
		\end{equation}
		
		Using Lemma \ref{lemma1}, we get that the first fraction is equal to $\frac{\frac{t_1}{t_1-1}}{\frac{t_2}{t_2-1}}$ and using $\left(\frac{1}{t_{i1}-1} \right)^2=t_i$ for the second fraction we get $\frac{t_2}{t_1}.$ Combining these, we obtain that the expression \ref{m2} is equal to $\frac{t_2-t_1}{t_1-1}=\frac{\sqrt{t}}{-\sqrt{t}}=-1$. Therefore, we see that $-1$ is a square in $K_\infty$ and hence $K_\infty$ contains a primitive fourth root of unity $\zeta_4$, which is equal to:
		
	\begin{equation}\label{zeta4}
	\left (\frac{t_{11}-1}{t_{21}-1} \right ) \prod\limits_{i=1}^2 \frac{(t_{1i1}-1)}{(t_{2i1}-1)} .
	\end{equation}	
	
In order to prove the second part of the theorem, we now proceed by induction. Assume that $\zeta_{2^m} $ is in $K_\infty$ and is  equal to the product of $\beta_i-1$ and $(\beta_j-1)^{-1}$ for some vertices $\beta_i$ and $\beta_j$. Then by Lemma \ref{lemma4e}, each term in this product, namely $\beta_i -1$ or $(\beta_j-1)^{-1}$, is of the form $\left [\frac{1}{(\beta_{i1}-1)(\beta_{i11}-1)(\beta_{i21}-1)} \right]^2$ or 
		$\left [\frac{1}{(\beta_{j1}-1)(\beta_{j11}-1)(\beta_{j21}-1)} \right]^{-2}$. This shows that $\zeta_{2^m} $ is a square in $K_\infty$ and therefore a primitive $2^{m+1}$th root of unity, $\zeta_{2^{m+1}} $ is in $K_\infty$and it has the desired form.
		
\end{proof}
		
	\begin{example} In this example we  explicitly show that $K_\infty$ contains a primitive $8$th root of unity. Consider the tree below:

	\scalebox{1.3}{	
			\begin{forest}
				my tree
				[$t$
				[$t_2$, name=t2
				[\color{white}$t_{22}$ , fill=white, circle, scale=0.2, draw
				[$t_{222}$, fill, circle, scale=0.1, draw
				[$t_{v2}$, fill, circle, scale=0.1, draw
				[$t_{v22}$, fill, circle, scale=0.1, draw]
				[$t_{v21}$, fill, circle, scale=0.1, draw]
				]
				[\color{red}$t_{v1}$, fill=red, circle, scale=0.2, draw
				[$t_{v12}$, fill, circle, scale=0.1, draw]
				[$t_{v11}$, fill, circle, scale=0.1, draw]
				]]
				[\color{blue}$t_{221}$, fill=blue, circle, scale=0.2, draw
				[$t_{s2}$, fill, circle, scale=0.1, draw
				[$t_{s22}$, fill, circle, scale=0.1, draw, ]
				[\color{purple}$t_{s21}$, fill=purple, circle, scale=0.2, draw]
				]
				[\color{red}$t_{s1}$, fill=red, circle, scale=0.2, draw
				[$t_{s12}$, fill, circle, scale=0.1, draw]
				[\color{purple}$t_{s11}$, fill=purple, circle, scale=0.2, draw]
				]
				]
				]
				[\color{green}$t_{21}$,  fill =green, circle, scale=0.2, draw
				[$t_{212}$, fill, circle, scale=0.1, draw
				[$t_{n2}$, fill, circle, scale=0.1, draw
				[$t_{n22}$, fill, circle, scale=0.1, draw]
				[$t_{n21}$, fill, circle, scale=0.1, draw]
				]
				[\color{red}$t_{n1}$ , fill=red, circle, scale=0.2, draw
				[$t_{n12}$, fill, circle, scale=0.1, draw]
				[$t_{n11}$, fill, circle, scale=0.1, draw]
				]
				]
				[\color{blue}$t_{211}$ , fill=blue, circle, scale=0.2, draw
				[$t_{m2}$, fill, circle, scale=0.1, draw
				[$t_{m22}$, fill, circle, scale=0.1, draw]
				[\color{purple}$t_{m21}$, fill=purple, circle, scale=0.2, draw]
				]
				[\color{red}$t_{m1}$, fill=red, circle, scale=0.2, draw
				[$t_{m12}$, fill, circle, scale=0.1, draw]
				[\color{purple}$t_{m11}$, fill=purple, circle, scale=0.2, draw]
				]
				]]]
				[$t_1$, name=t1
				[\color{white}$t_{12}$,  fill =white, circle, scale=0.2, draw
				[$t_{122}$, fill, circle, scale=0.1, draw
				[$t_{\ell2}$ , fill, circle, scale=0.1, draw
				[$t_{\ell22}$, fill, circle, scale=0.1, draw]
				[$t_{\ell21}$, fill, circle, scale=0.1, draw]
				]
				[\color{red}$t_{\ell1}$,  fill =red, circle, scale=0.2, draw
				[$t_{\ell12}$, fill, circle, scale=0.1, draw]
				[$t_{\ell11}$, fill, circle, scale=0.1, draw]
				]]
				[\color{blue}$t_{121}$, fill=blue, circle, scale=0.2, draw
				[$t_{k2}$ , fill, circle, scale=0.1, draw
				[$t_{k22}$, fill, circle, scale=0.1, draw]
				[\color{purple}$t_{k21}$, fill=purple, circle, scale=0.2, draw]
				]
				[\color{red}$t_{k1}$ , fill =red, circle, scale=0.2, draw
				[$t_{k12}$, fill, circle, scale=0.1, draw]
				[\color{purple}$t_{k11}$, fill=purple, circle, scale=0.2, draw]
				]
				]
				]
				[\color{green}$t_{11}$, fill =green, circle, scale=0.2, draw
				[$t_{112}$, fill, circle, scale=0.1, draw
				[$t_{j2}$ , fill, circle, scale=0.1, draw
				[$t_{j22}$, fill, circle, scale=0.1, draw]
				[$t_{j21}$, fill, circle, scale=0.1, draw]
				]
				[\color{red}$t_{j1}$ , fill =red, circle, scale=0.2, draw
				[$t_{j12}$ , fill, circle, scale=0.1, draw ]
				[$t_{j11}$ , fill, circle, scale=0.1, draw]
				]
				]
				[\color{blue}$t_{111}$, fill =blue, circle, scale=0.2, draw
				[$t_{i2}$, fill, circle, scale=0.1, draw
				[$t_{i22}$, fill, circle, scale=0.1, draw]
				[\color{purple}$t_{i21}$, fill=purple, circle, scale=0.2, draw]
				]
				[\color{red}$t_{i1}$, fill =red, circle, scale=0.2, draw
				[$t_{i12}$, fill, circle, scale=0.1, draw]
				[\color{purple} $t_{i11}$, fill=purple, circle, scale=0.2, draw] 
				]
				]]]
				]	
			 	\end{forest}}
 
 The equation~\ref{zeta4} describes $\zeta_4$ in terms of the vertices that are marked above by green and blue. Using Lemma~\ref{lemma4e} for each blue vertex $\beta$, we write $(\beta-1)$  as a product of $( \alpha -1)$ where 
 \begin{itemize}
		\item $\alpha$ runs through the red and two purple vertices above and directly connected to it.
 \end{itemize}
 Similarly, using Lemma \ref{lemma4e} for each green vertex $\beta$, we write $(\beta-1)$ as a product of $( \alpha -1)$ where 
 \begin{itemize}
		\item $\alpha$ runs through the blue and two red vertices above and directly connected to it.
 \end{itemize}
 Hence using the marked vertices, we find a primitive $8$th root of unity in $K_{\infty}$.
\end{example}

\section{Discriminant formula for the iterates of $f$}\label{sec:disc}
Let $f(x)$ be a rational map in $k(x)$.  Then, we write $f(x)=g(x)/h(x)$ where $g,h \in k[x]$ are relatively prime as polynomials in $k[x]$.  We define the discriminant of a rational function $g(x)/h(x)-t$ as the discriminant of the polynomial $g(x)-th(x)$ viewed as a polynomial in $x$ in $k(t)$.  In other words,
\[\Delta_x(f(x)-t):=\Delta_x(g(x)-th(x)).\]

In this section, we would like to find a formula for the discriminant of the iterates of the rational function $f(x)=1/(x-1)^2$. We  use this formula in the proof of Theorem~\ref{Frattini index} which is key in the proof of the main theorem.

Now, let $f(x)=\frac{g_1(x)}{h_1(x)}$ then $g_1(x)=1,\; h_1(x)=(x-1)^2$.  Suppose $f^{n}(x)=\frac{g_n(x)}{h_n(x)}$ for $n\geq 1$.  Then
\[
 f^{n}(x)=f(f^{n-1}(x))=\frac{1}{(f^{n-1}(x)-1)^2}=\frac{h_{n-1}^2(x)}{(g_{n-1}(x)-h_{n-1}(x))^2}.
\]

Hence, $g_n(x)=h_{n-1}^2(x)$ and $h_n(x)=(g_{n-1}(x)-h_{n-1}(x))^2.$

The discriminant formula for the iterates of the rational functions are given in \cite{CH} by
\[
\Delta_n:=\Delta_x(g_n(x)-th_n(x))=\pm \frac{l_n^{\epsilon_n+m_n-\delta_n-2}D_n^{\delta_n}}{l(h_n)^{\delta_n-q_n}\Res(g_n,h_n)}\prod_{r\in \cR_{f^{n}}}(g_n(r)-th_n(r))^{m_r}
\]

where $l_n=l_x(g_n(x)-th_n(x))$ and $l_x$ denotes the leading coefficient as a polynomial of $x$.  Furthermore, $m_n=\deg_x (g_n(x)-th_n(x))$,
$D_n=l(h_n(x)g_n'(x)-g_n(x)h_n'(x))$, $\epsilon_n=\deg (h_n(x))$, $\delta_n=\deg (g_n(x))$, 
 and $q_n=\deg(h_n(x)g_n'(x)-g_n(x)h_n'(x))$.  Finally, let
 \[
   \cR_{f^{n}}=\{r\in\overline{k} :(h_ng_n'-g_nh_n')(r)=0\} 
 \]
be the set of ramification points of $f^{n}$  
and $m_r$ be the multiplicity of $r \in \cR_{f^{n}}$.

Note that $\cR_f=\{r\in \overline{k}: -2(r-1)=0\}=\{1\}$ and $m_1=1$.\\
Our main goal in this subsection is to prove Proposition \ref{discriminant} and we  need the lemmas below to complete the proof.
\begin{proposition}\label{discriminant}
	For $n\geq 1$, \;$\Delta_n=c_nt^{a_n}(1-t)^{b_n}$ where $a_n,b_n$ are nonnegative integers and $c_n$ is a power of $2$ up to sign.
\end{proposition}

\begin{proof}  First of all, we would like to find the finite primes of $\overline{k}(t)$ dividing $\Delta_n$ and these terms can come from $\prod_{r\in \cR_{f^{n}}}(g_n(r)-th_n(r))^{m_r}$ and $l_n$.
	
Since the only branch points of of $f^{n}$ are $\{0,1,\infty\}$, $f^{n}(r)=\displaystyle\frac{g_n(r)}{h_n(r)}=0,1,\mbox{ or } \infty$ for any $r\in \cR_{f^{n}}$.  

Now, if $f^{n}(r)=0$ then $g_n(r)=0$ and therefore we get a term $[(-t)h_n(r)]^{m_r}$.
If $f^{n}(r)=\infty$, then $h_n(r)=0$ and so we get a term $g_n(r)^{m_r}$.
Finally, if $f^{n}(r)=1$ then $g_n(r)=h_n(r)$ and hence we have $(1-t)^{m_r}h_n(r)^{m_r}$.
		
Now we analyze $l_n=l_x(g_n(x)-th_n(x))=l_x(h_{n-1}^2(x)-t(g_{n-1}(x)-h_{n-1}(x))^2)$. We have 3 cases.\\ 

\textbf{Case 1:} $\deg (g_{n-1}(x))<\deg(h_{n-1}(x))$.  Then, $l_n=l_x(h_{n-1}^2(x)-th_{n-1}^2(x))=l_x(h_{n-1}^2(x))(1-t)$.\\

\textbf{Case 2:} $\deg (g_{n-1}(x))=\deg(h_{n-1}(x))$. We have $l_n=l_x(h_{n-1}^2(x))$.\\
 
\textbf{Case 3:} $\deg (g_{n-1}(x))>\deg(h_{n-1}(x))$.  Here, we obtain $l_n=-tl_x(g_{n-1}^2(x))$.\\

Therefore, we get $\Delta_n=c_nt^{a_n}(1-t)^{b_n}$ where $a_n$ and $b_n$ can be calculated. 
\end{proof}
 In Lemma \ref{lglh}, Lemma \ref{ramificationpoints} and Lemma \ref{resultant}, we prove that $l(h_n), l(g_n), D_n, \Res(g_n,h_n)$ and $g_n(r), h_n(r)$ for $r\in \cR_f$ are all powers of 2 up to sign which imply $c_n$ is a power of $2$ up to sign.

\begin{lemma}\label{lglh}
	For $n\geq 4$, 
	\[
	l(g_n)=\begin{cases}  2^4l(g_{n-3})^4l(h_{n-3})^4  & \text{ if } n\equiv 1 \pmod 3,\\
	                      l(g_{n-3})^8  &\text{ if } n\equiv 2 \pmod 3, \\
	                      l(g_{n-3})^8  &\text{ if } n\equiv 0 \pmod 3, 
	\end{cases}
	\]
and for $n\geq 3$

  	\[
  l(h_n)=\begin{cases}  2^2l(g_{n-2})^2l(h_{n-2})^2  & \text{ if } n\equiv 0 \pmod 3,\\
  l(g_{n-2})^4  &\text{ if } n\equiv 1 \pmod 3, \\
  l(g_{n-2})^4  &\text{ if } n\equiv 2 \pmod 3, 
  \end{cases}
  \]	
	
\end{lemma}

\begin{proof} We begin with by comparing the degrees of $g_n$ and $h_n$.
	
	Since $g_1(x)=1$, $h_1(x)=(x-1)^2$ we have $\deg g_1<\deg h_1$ and 
	$g_2(x)=(x-1)^4$, $h_2(x)=(1-(x-1)^2)^2$ give $\deg g_2=\deg h_2$.
	
Then	$g_3(x)=(1-(x-1)^2)^4$, $h_3(x)=\big[(x-1)^4\big(1-(x-1)^2\big)^2\big]^2$, $\deg h_3 < \deg g_3$.  Furthermore, $g_4=h_3^2$, $h_4=(g_3-h_3)^2$ imply $\deg g_4=2\deg h_3$ and $\deg h_4=2\deg g_3$.  Hence, we obtain $\deg g_4< \deg h_4$.  We can prove by induction that
  	
	\begin{align}\label{degree}
	               \deg g_n < \deg h_n \;  & \text{ if } n\equiv 1 \pmod 3, \nonumber\\
	               \deg g_n= \deg h_n\; &\text{ if } n\equiv 2 \pmod 3, \\  
	               \deg g_n >\deg h_n\;  &\text{ if } n\equiv 0 \pmod 3, \nonumber
	\end{align}
 
 Now, $l(g_n)=l(h_{n-1})^2=\big(l(g_{n-2}-h_{n-2})\big)^4$	and
 $l(g_{n-2}-h_{n-2})=l\big(h_{n-3}^2-l(g_{n-3}-h_{n-3})^2\big)=l(-g_{n-3}^2+2g_{n-3}h_{n-3})$ which give $l(g_n)=l(-g_{n-3}^2+2g_{n-3}h_{n-3})^4$.
 
 By comparing the degrees of $g_m$ and $h_m$ using the equation \ref{degree}, we get
 
 	\[
 l(-g_{m}^2+2g_{m}h_{m})=\begin{cases}  2l(g_{m})l(h_{m})  & \text{ if } m\equiv 1 \pmod 3,\\
 l(g_{m})^2  &\text{ if } m\equiv 2 \pmod 3, \\
 -l(g_{m})^2  &\text{ if } m\equiv 0 \pmod 3, 
 \end{cases}
 \]
 giving us the formula for $l(g_n)$ in Lemma \ref{lglh}.  For $n=4$, $l(g_4)=2^4l(g_1)^4l(h_1)^4=2^4$.\\
 
 In a very similar fashion, we can compute $l(h_n):$
 \begin{align*}
            l(h_n)&=l\big((g_{n-1}-h_{n-1})^2\big)\\
                  &=l\big(h_{n-2}^2-(g_{n-2}-h_{n-2})^2\big)^2\\
                  &=l\big(2g_{n-2}h_{n-2}-g_{n-2}^2\big)^2
 \end{align*}
 
  By comparing the degrees of $g_m$ and $h_m$ using the equation \ref{degree}, we get
 
 \[
    l\big(2g_{m}h_{m}-g_{m}^2\big)^2=\begin{cases}  \big(2l(g_{m})l(h_{m})\big)^2  & \text{ if } m\equiv 1 \pmod 3,\\
    l(g_{m})^4  &\text{ if } m\equiv 2 \pmod 3, \\
    l(g_{m})^4  &\text{ if } m\equiv 0 \pmod 3, 
 \end{cases}
 \]
 giving us the formula for $l(h_n)$ in Lemma \ref{lglh}.  For example, $l(g_1)=l(h_1)=1$ and $l(h_3)=2^2=4$.  Note that $h_3(x)=\big((x-1)^4-(1-(x-1)^2)^2\big)^2=\big(2(x-1)^2-1\big)^2$.\\
	
\end{proof}

\begin{lemma}\label{ramificationpoints}
	For  $r \in \cR_{f^{n}}$,\; $g_n(r)$ and $h_n(r)$ are powers of $2$ up to sign.
	
\end{lemma}
\begin{proof} Note that for $n=1$, $\cR_f=\{r\in \overline{k}: -2(r-1)=0\}=\{1\}$ and $g_1(1)=1$ and $h_1(1)=0$.  Now we compute $h_ng_n'-g_nh_n'$ where $g_n=h_{n-1}^2$ and $h_n=(g_{n-1}-h_{n-1})^2$.
	\begin{align}\label{res.eq.}
	  h_ng_n'-g_nh_n'&=2(g_{n-1}-h_{n-1})^2h_{n-1}h_{n-1}'-2h_{n-1}^2(g_{n-1}-h_{n-1})(g_{n-1}'-h_{n-1}')\nonumber\\
	                &=2(g_{n-1}-h_{n-1})h_{n-1}[(g_{n-1}-h_{n-1})h_{n-1}'-h_{n-1}(g_{n-1}'-h_{n-1}')]\\
	                &=2(g_{n-1}-h_{n-1})h_{n-1}[g_{n-1}h_{n-1}'-h_{n-1}g_{n-1}']\nonumber
	\end{align}
	
	For $n=2$ since $g_2=(x-1)^4$, $h_2(x)=(1-(x-1)^2)^2$ we get $\cR_f=\{0,1,2\}$ from the above computation.  Then $g_1(0)=g_1(1)=g_1(2)=1$ and $h_1(0)=1$, $h_1(1)=0$, $h_1(2)=1$. Also, $g_2(0)=h_1(0)^2=1$, $g_2(1)=h_1(1)^2=0$, $g_2(2)=h_1(2)^2=1$ and $h_2(r)=(g_1(r)-h_1(r))^2=0$ or $1$ for $r\in \cR_f=\{0,1,2\}$.  Now assume that for any $m<n$ and any $r\in \cR_{f^{n}}$, $g_m(r)$ and $h_m(r)$ are powers of $2$ up to sign.  Since $g_n=h_{n-1}^2$ then $g_n(r)$ is a power of $2$ up to sign by induction.  On the other hand, since the only branch points of $f^{n}$ are $0,1,\infty$ and $h_n(r)=(g_{n-1}(r)-h_{n-1}(r))^2=(g_{n-1}(r))^2$ or $(h_{n-1}(r))^2$ or $0$ and the result follows by induction. 
\end{proof}

\begin{lemma}\label{resultant}
	For $n\geq 1,\; \Res(g_n,h_n)$ and $D_n=l(h_ng_n'-g_nh_n')$ are powers of $2$ up to sign.
\end{lemma}

\begin{proof}  The resultant of $g_n$ and $h_n$ is defined by
	\[
	  \Res(g_n,h_n)=(-1)^{\delta_n \epsilon_n} l(h_n)^{\delta_n}\prod_{j=1}^{\epsilon_n }g_n(\theta_j)
	\]
	where $\epsilon_n=\deg (h_n(x))$, $\delta_n=\deg (g_n(x))$ and $\theta_j$'s are the roots of $h_n$.  Now $h_n(\theta_j)=0$ gives $f^{n}(\theta_j)=0$.  Since $f^{-1}(\infty)=\{1\}$, $f^{(n-1)}(\theta_j)=1$ and $1$ is a ramification point and hence $\theta_j$ is also a ramification point i.e. $\theta_j\in\cR_{f^{n}}$.  Therefore, by Lemma \ref{ramificationpoints} $g_n(\theta_j)$ is a power of $2$ up to sign.  We already showed that $l(h_n)$ is a power of $2$ up to sign in Lemma \ref{lglh} and so the result follows for $\Res(g_n,h_n)$.  
	
	Let us prove $D_n=l(h_ng_n'-g_nh_n')$ is also a power of $2$ up to sign.  For $n=1$ and $n=2$, $D_1=l(-2(x-1))=-2$ and $D_2=l(4(x-1)^3(1-(x-1)^2))$.  Assume for induction $D_{n-1}$ is a power of $2$ up to sign.  From the Equation \ref{res.eq.}
	\[
	  D_n=2l(h_{n-1})l(g_{n-1}-h_{n-1})D_{n-1}.
	\]
	
	Each term in $D_n$ is a power of $2$ up to sign and therefore the result follows.

\end{proof}	

\section{The proof of Theorem 1.1}\label{sec:main}
In this section we prove some explicit results on the specialization of the map $f(x)=\frac{1}{(x-1)^2}$. For any $n\geq 1$,
Hilbert’s Irreducibility Theorem implies that there exists a non-empty Zariski-open set $\mathcal{H}_n=\mathcal{H}_n(f) \subseteq
  \PP^1(k)$, called a \emph{Hilbert set}, such that specializing the parameter $t$ to $a\in \mathcal{H}_n$ does not change the Galois
 group. However the existence of elements that are in $\mathcal{H}_n$ for all $n$ is not guaranteed by Hilbert Irreducibility. In this section we determine explicit elements $a\in
  \mathcal{H}_n(f)$ for all $n$.

 We use the next proposition, whose proof can be found in \cite{arithmeticbasilica}, to show Theorem \ref{Frattini index}.
  \begin{proposition}\label{prop:Frattini}\cite[Proposition 5.2]{arithmeticbasilica}
Let $G$ be a finite or profinite group. Let $\Phi$ be the Frattini subgroup of $G$. Let $H$ be a subgroup of $G$ such that $\Phi H=G$, then $H=G$.
\end{proposition}

Remember that $K$ denotes the field $k(t)$ and $K_{n}$ denotes the splitting field of $f^{n}(x)-t$ over $K$ for $n \geq 1$. We defined earlier that $K_{\infty}:=\bigcup_{n} K_n$.  
 
 \begin{proposition}\label{K-infty}
The field $K_{\infty}$ (in fact $K_5$) contains the field $L=K(\sqrt{t},\sqrt{t-1},\zeta_8)$.
\end{proposition}
\begin{proof}
By Lemma \ref{lemma4e} we have $\sqrt{t} $ and $ \sqrt{t-1}$ are in $K_{\infty}$. In fact, since $t=\frac{1}{(t_{1}-1)^2}$, $\sqrt{t} $ is in $K_1$. Similarly from the expression given in Lemma~\ref{lemma4e}, $\sqrt{t-1}$ is in $K_3$. By Theorem \ref{unity} we have $\zeta_8 \in K_\infty$. Hence $K_\infty$ contains $L$.
In particular, using the construction in the proof of Theorem \ref{unity} it is possible to see that  $\zeta_8 \in K_5$. 
\end{proof}

\begin{theorem}\label{Frattini index}
The Frattini subgroup $\Phi(\ga)$ of $G^{\text{arith}}$ has index $16$ in $G^{\text{arith}}$ and it is the subgroup of $G^{\text{arith}}$ fixing the field $L=K(\sqrt{t},\sqrt{t-1},\zeta_8)$. 
\end{theorem}
\begin{proof}
Let $K=k(t)$ and $H$ be a maximal subgroup of $\ga$. Since $\ga$ is a pro-2 group, the index of $H$ in $\ga$ is 2. Let $K_H=K(\sqrt{a(t)})$ denote the quadratic extension of $K$ fixed by $H$ where $a(t)$ is a squarefree polynomial in $k[t]$ and hence $\Gal(K_\infty/K_H)=H$. Since any prime dividing $a(t)$ ramifies in $K_H$, it ramifies in $K_\infty$ hence ramifies in  $K_n=K(f^{n}(x)-t)$ for some $n\geq 1$. By Proposition \ref{discriminant}, $a(t)=ct^i(1-t)^j$ for some $i,j \in \{0,1\}$ and $c \in k$. 

Specializing at $t=-1$ we see that $a(-1)$ ramifies in some extension $k_n=k(f^{n}(x)+1)$ of $k$. By Proposition~\ref{discriminant},  the discriminant of $k_n$ is $(\pm 2)^m$ for some $m$ and hence $c$ has to be $\pm 2^m$ for some $m$.

Therefore  $K_H$ is contained in $L=K(\sqrt{t}, \sqrt{1-t}, \zeta_8)$. This implies that $\Gal(K_\infty /L)$ is a subgroup of $H$, where $H$ is any maximal subgroup of $\ga$, hence $\Gal(K_\infty /L) $ is a subgroup of $\Phi(\ga)$.

On the other hand, $$\Gal(K_\infty/L)=\Gal(K_\infty/K(\sqrt{t})) \cap \Gal(K_\infty/K(\sqrt{1-t})) \cap \Gal(K_\infty/K(\sqrt{2})) \cap \Gal(K_\infty/K(i)) 
$$ and each of these subgroups on the right are of index two in $\ga$ hence they are maximal subgroups of $\ga$ and therefore $\Phi(\ga)$ is contained in $\Gal(K_\infty/L)$.  

This shows that $\Phi(\ga)=\Gal(K_\infty/L)$ and since the degree of the extension $L/K$ is $16$, the claim is proved.
\end{proof}

\begin{theorem}\label{main}
Let $a \in k$ and $K_{\infty,a}:=\bigcup_{n} K_{n,a}$ where $K_{n,a}$ is the splitting field of $f^{n}(x)-a$ over $k$ for $n \geq 1$.  Let $G_a$ be the inverse limit of $G_{n,a}:=\Gal(K_{n,a}/k)$. The following are equivalent:
\begin{enumerate} \item The degree $| k(\sqrt{a},\sqrt{a-1},\zeta_8):k|=16$,  \label{1}
\item $G_a=G^{\text{arith}}$,\label{2}
\item $G_{5,a}=G^{\text{arith}}_5$ where $G_{5,a}$ denotes the Galois group of $\Gal(K_{5,a}/\QQ)$.\label{3}\
\end{enumerate}
\end{theorem}
\begin{proof}
We  start by showing that (\ref{1}) implies (\ref{2}). Define $L_a=k(\sqrt{a},\sqrt{a-1},\zeta_8)$. We calculate $\Delta_1=\Delta(f(x)-a)= 4a$ and $\Delta_2=\Delta(f^{2}(x)-a)=\pm 2^8 (a)^3(a-1)$. Then $K_{\infty,a}$ contains $\sqrt{a}$ and $\sqrt{a-1}$. By the construction in section~\ref{sec:roots}, we also obtain $\zeta_8$ in $K_{\infty,a}$. Hence $L_a$ is a subset of $K_{\infty,a}$. Since $\Gal(L_a/k)=16$, the Galois group $\Gal(L_a/k)$ consists of the automorphisms
\[ \sqrt{\Delta_1} \mapsto \pm \sqrt{\Delta_1}, \quad  \sqrt{\Delta_2} \mapsto \pm \sqrt{\Delta_2}, \quad \zeta_8 \mapsto \pm\zeta_8, \pm\zeta_8^3. \]
We can lift each of these maps to $G_a$. Hence $G_a$ contains an element from each coset of the Frattini subgroup $\Phi(G^{\text{arith}})$ in $G^{\text{arith}}$. Since $G_a$ is a subgroup of $G^{\text{arith}}$, we have $G^{\text{arith}}=G_a\Phi(G^{\text{arith}})$. By Proposition~\ref{prop:Frattini}, $G_a=G^{\text{arith}}$.

It is obvious that (\ref{2}) implies (\ref{3}).  In order to prove (\ref{3}) implies (\ref{1}), we first claim that $|G^{\text{arith}}_5/G_5|=4$. 

 Let $N_5$ be the normalizer of $G_5$ in $W_5$. A Magma computation shows that $|N_5/G^{}_5|=4$. Since $G_5$ is a normal subgroup of $G^{\text{arith}}_5$,
we have $G^{\text{arith}}_5$ is a subgroup of $N_5$. 
Hence $|{G^{\text{arith}}_5}:G_{5}|\leq 4$. and since we know already that $\zeta_8$ is contained in $K_{\infty}$, we obtain
 $|G^{\text{arith}}_5:G_5|=4$.
 
 In this step, we  describe the elements of $G_5^{\text{arith}}$.  Since $L_a\subseteq K_5$ by Proposition \ref{K-infty}, we will find $16$ automorphisms in $\Gal(L_a/k)$.  Let $G_5^{\text{arith}}=\langle a_1,a_2,w_1,w_2 \rangle$ where $w_1$ and $w_2$ are defined as follows:
 
 \begin{align*}
   s_1&=(-,\sigma)\in W_2,\; s_2=(s_1,s_1) \in W_3, \; s_3=(s_2,-) \in W_4 \text{ and } w_1=(s_3,-) \in W_5\\
   u_{11}&=(-,\sigma) \in W_2,\; u_1=(u_{11},-) \in W_3,\; u_2=\sigma \in W_3, \; u=(u_1,u_2) \in W_4,\\
   v_{21}&=(-,\sigma) \in W_2,\; v_2=(v_{21},-) \in W_3, \; v_1=(-,\sigma) \in W_3,\; v=(v_1,v_2)\in W_4 \text{ and } w_2=(u,v) \in W_5
 \end{align*}

We include the diagram below to demonstrate the action of the elements $w_1, w_2$.

\begin{center}
 \scalebox{1.3}{	
	\begin{forest}
 		my tree
 		[$t$, fill, circle, scale=0.1, name=t
 		[$t_2$, fill,circle, scale=0.1, name=t2
 		[$t_{22}$ , fill, circle, scale=0.1, draw, name=t22
 		[$t_{222}$, fill, circle, scale=0.1, draw
 		[$t_{v2}$, fill, circle, scale=0.1, draw
 		[$t_{v22}$, fill, circle, scale=0.1, draw]
 		[$t_{v21}$, fill, circle, scale=0.1, draw]
 		]
 		[$t_{v1}$, fill, circle, scale=0.1, draw
 		[$t_{v12}$, fill, circle, scale=0.1, draw]
 		[$t_{v11}$, fill, circle, scale=0.1, draw]
 		]]
 		[\color{blue}$t_{221}$, fill, circle, scale=0.1, draw
 		[$t_{s2}$, fill, circle, scale=0.1, draw
 		[$t_{s22}$, fill, circle, scale=0.1, draw, name=s22]
 		[\color{purple}$t_{s21}$, fill, circle, scale=0.1, draw, name=s21]
 		]
 		[\color{red}$t_{s1}$, fill, circle, scale=0.1, draw
 		[$t_{s12}$, fill, circle, scale=0.1, draw]
 		[\color{purple}$t_{s11}$, fill, circle, scale=0.1, draw]
 		]
 		]
 		]
 		[$t_{21}$,  fill, circle, scale=0.1, draw, name=t21
 		[$t_{212}$, fill, circle, scale=0.1, draw, name=t212
 		[$t_{n2}$, fill, circle, scale=0.1, draw, name=tn2
 		[$t_{n22}$, fill, circle, scale=0.1, draw, name=n22]
 		[$t_{n21}$, fill, circle, scale=0.1, draw, name=n21]
 		]
 		[$t_{n1}$ , fill, circle, scale=0.1, draw, name=tn1
 		[$t_{n12}$, fill, circle, scale=0.1, draw]
 		[$t_{n11}$, fill, circle, scale=0.1, draw]
 		]
 		]
 		[\color{blue}$t_{211}$ , fill, circle, scale=0.1, draw, name=t211
 		[$t_{m2}$, fill, circle, scale=0.1, draw
 		[$t_{m22}$, fill, circle, scale=0.1, draw, name=m22]
 		[\color{purple}$t_{m21}$, fill, circle, scale=0.1, draw, name=m21]
 		]
 		[\color{red}$t_{m1}$, fill, circle, scale=0.1, draw
 		[$t_{m12}$, fill, circle, scale=0.1, draw]
 		[\color{purple}$t_{m11}$, fill, circle, scale=0.1, draw]
 		]
 		]]]
 		[$t_1$, fill, circle, scale=0.1, name=t1
 		[$t_{12}$,  fill, circle, scale=0.1, draw, name=t12
 		[$t_{122}$, fill, circle, scale=0.1, draw, name=t122
 		[$t_{\ell2}$ , fill, circle, scale=0.1, draw
 		[$t_{\ell22}$, fill, circle, scale=0.1, draw]
 		[$t_{\ell21}$, fill, circle, scale=0.1, draw]
 		]
 		[$t_{\ell1}$,  fill, circle, scale=0.1, draw
 		[$t_{\ell12}$, fill, circle, scale=0.1, draw]
 		[$t_{\ell11}$, fill, circle, scale=0.1, draw]
 		]]
 		[\color{blue}$t_{121}$, fill, circle, scale=0.1, draw, name=t121
 		[$t_{k2}$ , fill, circle, scale=0.1, draw
 		[$t_{k22}$, fill, circle, scale=0.1, draw]
 		[\color{purple}$t_{k21}$, fill, circle, scale=0.1, draw]
 		]
 		[\color{red}$t_{k1}$ , fill, circle, scale=0.1, draw
 		[$t_{k12}$, fill, circle, scale=0.1, draw]
 		[\color{purple}$t_{k11}$, fill, circle, scale=0.1, draw]
 		]
 		]
 		]
 		[$t_{11}$, fill , circle, scale=0.1, draw, name=t11
 		[$t_{112}$, fill, circle, scale=0.1, draw, name=t112
 		[$t_{j2}$ , fill, circle, scale=0.1, draw
 		[$t_{j22}$, fill, circle, scale=0.1, draw,name=buraya]
 		[$t_{j21}$, fill, circle, scale=0.1, draw, name=burdan]
 		]
 		[$t_{j1}$ , fill , circle, scale=0.1, draw
 		[$t_{j12}$ , fill, circle, scale=0.1, draw ]
 		[$t_{j11}$ , fill, circle, scale=0.1, draw]
 		]
 		]
 		[\color{blue}$t_{111}$, fill , circle, scale=0.1, draw, name=t111
 		[$t_{i2}$, fill, circle, scale=0.1, draw
 		[$t_{i22}$, fill, circle, scale=0.1, draw, name=where in]
 		[\color{purple}$t_{i21}$, fill, circle, scale=0.1, draw, name=where out]
 		]
 		[\color{red}$t_{i1}$, fill , circle, scale=0.1, draw
 		[$t_{i12}$, fill, circle, scale=0.1, draw]
 		[\color{purple} $t_{i11}$, fill, circle, scale=0.1, draw]  		]
 		]]]
 		]	
 		\draw[-stealth] (where out) to[out=70, in=120] node [above]{\tiny$w_1$}(where in);
 		\draw[shorten <=0.2cm, shorten >=0.2cm, -stealth] (t121) to [out=30, in=150] node [above]{\tiny$w_2$}(t122);
 		\draw[shorten <=0.1cm, shorten >=0.1cm, -stealth] (tn1) to [out=30, in=150] node [above]{\tiny$w_2$}(tn2);
 		\draw[-stealth] (s21) to [out=70, in=120] node [above]{\tiny$w_2$}(s22);
 		\draw[-stealth] (where out) to[out=-70, in=-120] node [below]{\tiny$\;w_2\;$}(where in);
 		\draw[-stealth] (burdan) to[out=70, in=120] node [above]{\tiny$w_1$}(buraya);
 	 \end{forest}}
\end{center}
 
 Now, $a_1\restr{T_1}=\id$ so $a_1$ fixes $\sqrt{t}$ (or  $\sqrt{\Delta_1}$ ) and since $a_1\restr{T_2}=(\id,\sigma)$ is an odd permutation $ \sqrt{\Delta_2} \mapsto -\sqrt{\Delta_2}$ by $a_1$.  On the other hand, $a_2\restr{T_1}\neq\id$ so it sends  $ \sqrt{t} \mapsto -\sqrt{t}$ (or  $ \sqrt{\Delta_1} \mapsto -\sqrt{\Delta_1}$) and $a_2\restr{T_2}=(\id,\id)\sigma$ is an even permutation at level $2$ leaving $\sqrt{\Delta_2}$ fixed.  Also, as $a_1,\;a_2$ are in $G_{}$ which fixes $\overline{k}(t)$ they both fix $\zeta_8$.
 
 The final step in the proof is that there exist two distinct automorphisms fixing the first two levels and sending $\zeta_8 \mapsto \zeta_8 ^k$ for $k\in\{3,5,7\}$.

 As $w_1\restr{T_4}=\id$, it fixes $\sqrt{\Delta_1}$ and $\sqrt{\Delta_2}$ and it only moves $t_{11121}$ to $t_{11122}$ and $t_{11221}$ to $t_{11222}$ at level five.  By using Equation \ref{zeta4} and the fact that $\alpha_{1}-1=-(\alpha_{2}-1)$ for any vertex $\alpha$, we obtain $w_1(\zeta_8)=-\zeta_8=\zeta_8^5$.
 
 Similarly, we see that $w_2$ fixes $\sqrt{\Delta_1}$ and $\sqrt{\Delta_2}$ since $w_2\restr{T_2}=\id$ and it only moves $t_{121}$ to $t_{122}$ at level three.  By using Equation \ref{zeta4} and the fact that $t_{121}-1=-(t_{122}-1)$, we obtain $w_2(\zeta_8^2)=w_2(i)=-i$.  Hence, $w_2(\zeta_8)\neq \pm \zeta_8$.
 
 The restrictions of $a_1,a_2,w_1,w_2$ to $L_a/k$ give exactly $16$ automorphisms and therefore $|L_a:k|=16$.
\end{proof}

\section{ Further properties of the Group $G$}\label{furtherG}

We first prove that $G_2=W_2$ when $n\leq 2$.
    \begin{proposition}\label{G-E isom. n=2}
    	For $n\leq 2$, $G_n=W_n$. 
    \end{proposition}	
 \begin{proof} When $n=1$, $a_2\restr{T_1}=\sigma$ generates $W_1$. If $n=2$, then $W_2$ has order $8$. Both $a_1\restr{T_2}=(\id,\sigma)$ and $a_2\restr{T_2}=\sigma$ have order $2$. We have $a_3\restr{T_2}=(\id, \sigma)\sigma$ has order $4$. They together generate a subgroup of order $8$, hence $G_2=W_2$.
 \end{proof}	
	
	Now we describe $G_n$ when $n=3$.  First we introduce some notation.
	
	\begin{definition}\label{def:sgn3}\mbox{} 
	Let $w:=(\sigma_1,\sigma_2)\tau$ be in $W_3$ where $\sigma_i=(\sigma_{i1},\sigma_{i2})\tau_i$ for $i=1,2$.
	\begin{enumerate} 
  \item We define $ \sgn_3 :W_3\to \{\pm 1\}$ by
       \begin{equation}
             \sgn_3(w) = \sgn(\tau) \prod_{i=1}^2(\sgn(\tau_i) \sgn(\sigma_{i1})\sgn(\sigma_{i2})).
         \end{equation}
    where $\sigma_i \in W_2$.
  Here, $\sgn$ is the usual sign on $W_1$ via the identification $
     W_1\simeq S_2$ induced by the choice of labeling of the vertices.
 \item For $n> 3$ we define 
   \begin{equation}\label{eq:sgn3}
       \sgn_3:=\sgn_3\circ \pi_{n,3}:W_n\to \{\pm 1\}.
   \end{equation}
\end{enumerate}
\end{definition}

Now we  prove that $\sgn_3: W_3 \rightarrow \{\pm1\}$ is a homomorphism.

\begin{proposition}  The map 
	$\sgn_3: W_3 \rightarrow \{\pm1\}$ is a homomorphism.
\end{proposition}

\begin{proof}
	Let $s$ be the map $s: W_2 \to \{\pm 1\}$ defined by $s(\tau_1,\tau_2)\tau=\sgn(\tau)\prod_i\sgn(\tau_i)$. We first observe that $s$ is a homomorphism.
	
	Let $f,g \in W_2$. Then by definition of $f=(\sigma_1, \sigma_2)\tau$ and  $g=(\beta_1, \beta_2)\tau'$ where $\sigma_i , \beta_i \in W_1=S_2$ and $\tau, \tau' \in S_2$. Using the group operation in $W_2$, we get $f * g= (\sigma_1 \beta_1, \sigma_2\beta_2) \tau \tau'$ or $(\sigma_1 \beta_2, \sigma_2\beta_1) \tau \tau' $. In either case, using the fact that $\sgn$ is a homomorphism, we get $s(f*g)= \sgn(\tau ) \sgn(\tau') \prod_{i=1}^2
	\sgn(\sigma_i)\sgn(\beta_i)$ which is equal to $s(f)s(g).$

	Now let $f=(\sigma_1, \sigma_2)\tau$ and $ g=(\beta_1, \beta_2)\tau' \in W_3$ where  $\sigma_i , \beta_i \in W_2$ and $\tau, \tau' \in S_2$. Using the group operation in $W_3$, we get $f \star g = (\sigma_1, \sigma_2) * (\beta_1, \beta_2) \tau \tau'$ where $*$ denotes the group operation in $W_2$. Hence we get $f\star g= (\sigma_1 * \beta_1, \sigma_2*\beta_2)\tau \tau'$ or $(\sigma_1 * \beta_2 \sigma_2*\beta_1)\tau \tau'$. Using the facts that both $\sgn$ and $s$ are homomorphism we get $\sgn_3(f \star g)=\sgn(\tau)\sgn( \tau') \prod_{i=1}^2
	s(\sigma_i)s(\beta_i)= \sgn_3(f) \sgn_3(g).$
	\end{proof}

\begin{remark} Since the natural projection map $\pi_{n,3}:W_n \rightarrow W_3$ is also a homomorphism, we get that $\sgn_3: W_n \to \{\pm 1\}$ is a homomorphism.
 \end{remark}

 \begin{proposition}\label{G-E isom. n=3}
 	$G_3= \ker(\sgn_3: W_3 \rightarrow \{\pm1\})$.
 \end{proposition}	
	
	\begin{proof}

		Let $\pi_3: W \rightarrow W_3$ be the natural projection map. Then by definition, $G_3=\im (\pi_3\restr{G})$.  
	
	We know that $G_3=\langle a_1\restr{T_3}, a_2\restr{T_3}, a_3\restr{T_3} \rangle$ where $a_1, a_2, a_3$ defined as above.  Furthermore, $a_1\restr{T_3}=(\id,a_3\restr{T_2})$, $a_3\restr{T_2}=(a_2\restr{T_1},\id)\sigma$ and $a_2\restr{T_1}$ acts as $\sigma$ on $T_1$.  Hence,
	\begin{align*}
	\sgn_3(a_1\restr{T_3})&=s(a_3\restr{T_2})\\
	           &=\sgn(a_2\restr{T_1})\sgn (\sigma) \\
	           &=\sgn(\sigma)^2=(-1)^2.
	\end{align*}
	
	Similarly, $\sgn_3(a_2\restr{T_3})=\sgn_3(a_3\restr{T_3})=1$.  Thus, $G_3\leq W_3\cap\ker (\sgn_3)$ and the latter group is a subgroup of $W_3$ of index $2$. Since the order of $W_3=\Aut(T_3)=2^7$,  the order of $W_3\cap\ker (\sgn_3)$ is $64$.
	
	 In order to finish the proof, it is enough to show that the size of $G_3$ is $64$. 
Note that, by Lemma \ref{G-E isom. n=2}, $G_2=W_2$ and therefore $\#G_2=8.$ If the kernel of the natural projection map $(\pi_{3,2})_{\restr{G_3}}$ from $G_3$ to $ G_2$ is $8$ then we get $\#G_3=64$. Now, we  explicitly list these $8$ elements in the kernel $V_3$ of $(\pi_{3,2})_{\restr{G_3}}:G_3\rightarrow G_2$. 
	
	Recall that $a_1=(\id,a_3)$ and hence $a_1^2=(\id,a_3^2)$ where $a_3^2=[(a_2,\id)\sigma]^2=(a_2,a_2)$. We find that 
$a_1^2\restr{T_3}=(\id,a_3^2\restr{T_2})$ and $a_3^2\restr{T_2}=(a_2\restr{T_1},a_2\restr{T_1})$. This shows that $a_1^2$ is in $V_3$ and using the labeling of the leaves given in equation~\ref{eq:leafnumber} and the embedding given in ~\ref{eq:iota}, we find that $a_1^2\restr{T_3}=(26)(48)$.
  \begin{center}    
  \includegraphics{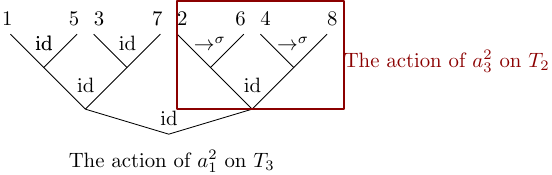}
  \end{center}

Doing a similar calculation, we find that $a_2^2\restr{T_3}=(37)(48)$. According to our labeling (as can be seen on the picture below) $a_2^2$ fixes every leaf in $T_2$ therefore $a_2^2$ is in $V_3$. Another way to see this is as follows: \[ a_2^2\restr{T_2}=(a_1\restr{T_1}, a_1\restr{T_1})=(\id,\id)=\id \]
 
	 \begin{center}   
   \includegraphics{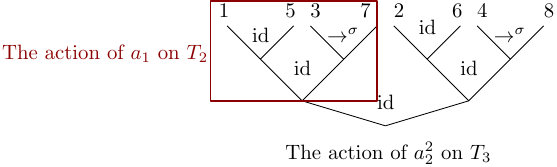}
     \end{center}

Notice that we need $8$ elements in the kernel $V_3$ and since $V_3$ is elementary abelian, we just need $3$ generators.  We already have $a_1^2, a_2^2$ and we need to find one more generator. We consider the element $a_3 a_1^2 a_3^{-1}$.

\begin{align*}
   a_3 a_1^2 a_3^{-1}\restr{T_3}&=(a_2\restr{T_2},\id)\sigma(\id,a_3^2\restr{T_2})\sigma(a_2^{-1}\restr{T_2},\id)=(a_2 a_3^2 a_2^{-1}\restr{T_2},\id) \\
                          a_2 a_3^2 a_2^{-1}\restr{T_2}&= (\id,a_1\restr{T_1})\sigma (a_2\restr{T_1},a_2\restr{T_1})(a_1^{-1},\id)\sigma =(a_2\restr{T_1},a_1 a_2  a_1^{-1}\restr{T_1}) \end{align*}
One then calculates $a_3 a_1^2 a_3^{-1}=(15)(37)$ and finds that $\# \langle a_3 a_1^2 a_3^{-1}, a_1^2, a_2^2\rangle=8$ and hence $\#G_3=\#G_2\cdot 8=64$.

Note that since the kernel is a normal subgroup and $a_1^2\in V_3=\ker((\pi_{3,2})_{\restr{G_3}}:G_3\rightarrow G_2)$, then $a_3 a_1^2 a_3^{-1}$ is in $V_3$ as well.		
\end{proof}

Now we calculate the orders of the generators for any $n\geq 1$ and show in particular that the generators $a_1,a_2,a_3$ of $G$ have infinite order.
\begin{proposition}\label{orders}
The orders of $a_1, a_2, a_3$ restricted to $T_n$, $\ord(a_i\restr{T_n})$, are $2^{m_{i,n}}$ where
\[
m_{1,n}=\begin{cases}  2\floor*{\frac{n}{3}}+1 & \text{ if }n\equiv 2 \pmod 3 ,\\
                       2\floor*{\frac{n}{3}}&\text{ otherwise}
\end{cases}
\]
\[
m_{2,n}=\begin{cases}  2\floor*{\frac{n}{3}} & \text{ if }n\equiv 0 \pmod 3 ,\\
2\floor*{\frac{n}{3}}+1&\text{ otherwise}
\end{cases}
\]
\[
m_{3,n}=\begin{cases}  2\floor*{\frac{n}{3}}+1 & \text{ if }n\equiv 1 \pmod 3 ,\\
2\floor*{\frac{n-1}{3}}+2&\text{ otherwise}
\end{cases}
\]
\end{proposition}
\begin{proof}
	Let $n=1$, then $a_1\restr{T_1}=\id$ implying $\ord(a_1\restr{T_1})=1$ and $m_{1,1}=0$.  We also know $a_2\restr{T_1}=a_3\restr{T_1}=\sigma$, hence $\ord(a_2\restr{T_1})=\ord(a_3\restr{T_1})=2$ and $m_{2,1}=m_{3,1}=1$.
	
	Now assume $m_{i,n-1}$ is given as above for all $1\leq i \leq 3$.  Since $a_1\restr{T_n}=(\id,a_3\restr{T_{n-1}})$, $\ord(a_1\restr{T_n})=\ord(a_3\restr{T_{n-1}})$.  Similarly, $a_2\restr{T_n}=(\id,a_1\restr{T_{n-1}}\sigma)$ and one can show that $\ord(a_2\restr{T_n})=2\ord(a_1\restr{T_{n-1}})$.  Also, we have 
	$a_3\restr{T_n}=(a_2\restr{T_{n-1}},\id)$ and $\ord(a_3\restr{T_n})=2\ord(a_2\restr{T_{n-1}})$. 
	
	 To summarize, we obtain the below relation:
	\begin{align*} \
	   m_{1,n}&=m_{3,n-1}\\
	   m_{2,n}&=m_{1,n-1}+1\\
	   m_{3,n}&=m_{2,n-1}+1
	\end{align*} 
	
	Hence, by induction
	
	\[
	m_{1,n}=m_{3,n-1}=\begin{cases}  2\floor*{\frac{n-1}{3}}+1 & \text{ if }n-1\equiv 1 \pmod 3 ,\\
	2\floor*{\frac{n-2}{3}}+2&\text{ otherwise}
	\end{cases}
	\]
	\[
	m_{2,n}=m_{1,n-1}+1=\begin{cases}  2\floor*{\frac{n-1}{3}}+2 & \text{ if }n-1\equiv 2 \pmod 3 ,\\
	2\floor*{\frac{n-1}{3}}+1&\text{ otherwise}
	\end{cases}
	\]
	\[
	m_{3,n}=m_{2,n-1}+1=\begin{cases}  2\floor*{\frac{n-1}{3}}+1 & \text{ if }n-1\equiv 0 \pmod 3 ,\\
	2\floor*{\frac{n-1}{3}}+2&\text{ otherwise}
	\end{cases}
	\] 
	
	If $n\equiv 2 \pmod 3$, then $\displaystyle \floor*{\frac{n-1}{3}}=\floor*{\frac{n}{3}}$.\\
	
	If $n\equiv 1 \pmod 3$, then $\displaystyle \floor*{\frac{n-1}{3}}=\floor*{\frac{n}{3}}$ and  $\displaystyle \floor*{\frac{n-2}{3}}=\floor*{\frac{n-1}{3}}-1=\floor*{\frac{n}{3}}-1$.\\
	
	If $n\equiv 0 \pmod 3$, then $\displaystyle \floor*{\frac{n-1}{3}}=\floor*{\frac{n}{3}}-1$ and  $\displaystyle \floor*{\frac{n-2}{3}}=\floor*{\frac{n}{3}}-1$.\\
	
  These observations prove the given formulae.	
	
\end{proof}

The values of $m_{i,n}$ for $n\leq 18$ can be seen in Table~\ref{table:order}.	
\begin{equation}\label{table:order}
\begin{array}{ ||c|c|c|c|c|c|c|c|c|c|c|c|c|c|c|c|c|c|c|| } 
 \hline 
 n & 1 &2 & 3&4 &5 & 6&7 &8 &9 & 10& 11 & 12 & 13&14&15&16&17 & 18\\ 
 \hline \hline
  m_{1,n} & 0 &1 &2 &2 & 3&4 &4 &5 &6 &6 &7 &8 &8&9& 10 &10& 11 &12 \\ 
 \hline
  m_{2,n} & 1 & 1 & 2&3 &3 &4 &5 &5 &6 &7 &7 &8 &9&9&10&11 &11 &12 \\ 
 \hline 
m_{3,n}& 1 & 2 & 2& 3&4 &4 &5 & 6& 6& 7& 8&8 &9 &10 &10 &11 & 12 &12\\ 
 \hline
\end{array}
\end{equation}

\begin{corollary}
For any $i \in \{1,2,3\}$, the subgroup of $W$ generated by $a_i$ is isomorphic to $\ZZ_2$, the $2$-adic integers.
\end{corollary}

Next lemma can be used to produce elements in $G_n$.

\begin{lemma}\label{lem:compG}
For any $x\in G$, the element $(x,x)$ is in $G$.
\end{lemma}
\begin{proof}
We have $a_2^2=(a_1,a_1)$ and $a_3^2=(a_2,a_2)$. Since $G$ is generated by $a_1,a_2$, the result follows.
\end{proof}

\begin{definition}
For $i \in \{1,2,3\}$ , let $N_i$ be the normal closure of the subgroup of $G$ generated by $a_i$ in $G$. We denote the image of $N_i$ in $G_n$ under the natural projection map $\pi_n: W \rightarrow W_n$ as $N_{i,n}$.
\end{definition}

For $n\geq 2$, note that the group $G_n=N_{i,n}\cdot\langle a_j\restr{T_n}\rangle$ for $i \ne j$ and $i,j \in \{1,2,3\}$. Next proposition proves that $G/N_{i}$ has infinite order and it is pro-cyclic for all $i \in \{1,2,3\}$. In order to prove the proposition, we need Lemma~\ref{lem:N1}, Lemma~\ref{lem:N23}, and Lemma~\ref{lem:Gn1} which describe the iterated structure of $N_{i,n}$ for $i=1,2,3$.

\begin{proposition}\label{G/Ninfinite}
For $n \geq 2$ and for $i\in \{1,2,3\}$, the quotient $G_n/N_{i,n}$ is a cyclic group and
\begin{enumerate} 
\item   $|G_{n}/N_{i,n}| \leq 2|G_{n-1}/N_{i,n-1}| $.
\item $ |G_n/N_{3,n}| \geq |G_{n-1}/N_{2,n-1}|$,
\item 
$ |G_n/N_{2,n}| \geq |G_{n-1}/N_{1,n-1}|$,
\item $ |G_n/N_{1,n}| \geq 2\cdot |G_{n-1}/N_{3,n-1}|$. 
\end{enumerate}
\end{proposition}

\begin{proof}
Since $G_n$ is generated by $a_i\restr{T_n}$ and $a_j\restr{T_n}$ for any $ i\neq j $, the quotient $G_n/N_{i,n}$ is a quotient of  $\langle a_j\restr{T_n}\rangle$, hence it is cyclic.  

Let $V_n$ denote the kernel of the projection map $\pi_{n-1}: G_n \to G_{n-1}$.
Since $\pi_{n-1}(N_{i,n})$ is contained in $N_{i,n-1}$, we have a projection map 
$\psi_{n-1}: G_n/N_{i,n} \to  G_{n-1}/N_{i,n-1}$. We will show that the kernel of $\psi_{n-1}$ is isomorphic to the quotient $V_n/(V_n \cap N_{i,n})$. There is a natural injective map $V_n/(V_n \cap N_{i,n}) \to G_n/N_{i,n}$. In fact the image of this map is $ker(\psi_{n-1})$. We will show that it is surjective onto $ker(\psi_{n-1})$. Let $uN_{i,n}$ be in $\ker(\psi_{n-1})$. Then there is an $x \in N_{i,n}$ such that $\pi_{n-1}(x)=\pi_{n-1}(u)$ since $ \pi_{n-1}(u)$ is in $N_{i,n-1}$. Hence we obtain that $ux^{-1}$ is in $V_n$ which implies that $uN_{i,n}=ux^{-1}N_{i,n}$ which proves our claim.

The group $V_n/(V_n \cap N_{i,n})$ is at most of order two since it is a subgroup of the cyclic group $G_n/N_{i,n}$ and $V_n$ is an elementary abelian $2$-group. Hence the first claim follows.

Assume that $a_1^k\restr{T_n}$ is in $N_{3,n}$ for some $k$. Then
$a_1\restr{T_n}^k=(\id, a_3^k\restr{T_{n-1}}) \in N_{3,n}$ and by Lemma~\ref{lem:N23}, we have $a_3^k\restr{T_{n-1}} \in N_{2,n-1}$. Hence the order of $G_n/N_{3,n}$ is at least the order of  $G_{n-1}/N_{2,n-1}$. A similar argument holds to prove that  $|G_n/N_{2,n}| \geq |G_{n-1}/N_{1,n-1}|$. Assume $a_2^k\restr{T_n} \in N_{1,n}$ for some $k$. By Lemma~\ref{lem:N1}, $a_1^{k/2}$ is in $N_{3,n-1}$. This proves that $|G_n/N_{1,n}| \geq 2\cdot |G_{n-1}/N_{3,n-1}|$.
\end{proof}

 \begin{lemma}\label{lem:N1}
For any $n\geq 2$,  $N_{1,n}=N_{3,n-1} \times N_{3,n-1}$.
\end{lemma}
\begin{proof}
$N_{1,n}$ is generated by the conjugates of $a_1\restr{T_n}=(\id,a_3\restr{T_{n-1}})$ by the elements of $G_n$. One computes the conjugates of $a_1$ as $(u,\id)$ or $(\id,v)$ where $u,v$ are conjugates of $a_3\restr{T_{n-1}}$. 
\end{proof}

\begin{lemma}\label{lem:N23}
For $i=2,3$, we have 
 \[ N_{i} \cap (W\times W) \subset (N_{i-1} \times N_{i-1})\cdot\langle (a_i,a_i^{-1})\rangle \text{ and }   N_{i} \cap (\{\id \} \times W)  \subset (\{\id\} \times N_{i-1}). \]
\end{lemma}
\begin{proof}

We  prove it for $i=2$. It is similar for $i=3$. We follow the proof of \cite[Lemma~2.2.4(2)]{Pinkpolyn}.  Let $H=(N_{1} \times N_{1})\cdot\langle (a_2,a_2^{-1})\rangle$ and $\tilde{G}=(G\times G) \rtimes \langle \sigma \rangle$. 
First of all, $H$ is a normal subgroup of $\tilde{G}$ with an abelian quotient. We see this as follows: 

First of all, $N_{1} \times N_{1}$ is a normal subgroup of $\tilde{G}$ and the following homomorphism
 \[ (\langle a_2\rangle \times \langle a_2\rangle) \rtimes \langle \sigma\rangle \to \tilde{G}/(N_{1} \times N_{1}) \]
 is a surjection. Hence the abelian group $(\langle a_2\rangle \times \langle a_2\rangle) \rtimes \langle \sigma\rangle / \langle (a_2,{a_2}^{-1}) \rangle$ surjects onto $\tilde{G}/H$. 

Let $\tilde{H}:=H\cdot\langle a_2 \rangle$. Note that the quotient $\tilde{G}/H$ is abelian and $\tilde{H}$ contains $H$, hence $\tilde{H}$ is also a normal subgroup of $\tilde{G}$. Hence $N_{2}$ is contained in $\tilde{H}$.  We also see that $a_2$ is not in $H$ but ${a_2}^2=(a_1,a_1)$ is in $H$. Hence $N_{2} \cap (W\times W)$ is a subgroup of $H=\tilde{H} \cap (W\times W)$.

If $(u,v) \in N_{i} \cap (\{\id \} \times W)$, then $u=u'a_2^k$ and $v=v'a_2^{-k}$ by the first statement and $u=\id$ implies that $a_2^{k}$ is in $N_1$. Hence $v$ is in $N_1$.
\end{proof}

\begin{lemma}\label{lem:Gn1}
For any $n\geq 2$, we have $G_n \cap (W_{n-1}\times W_{n-1}) \simeq (N_{3,n-1}\times N_{3,n-1})\cdot \langle (a_1,a_1)\rangle$.
\end{lemma}
\begin{proof}
Follows from Lemma~\ref{lem:N1} and the fact that $G_n \simeq N_{1,n}\cdot\langle a_2\restr{T_n}\rangle$.
\end{proof}

\begin{remark}
Our computations in Magma shows that $|G_n:N_{1,n}|=2^{\floor*{\frac{n+1}{2}}}$ and $|G_n:N_{3,n}|=2^{\floor*{\frac{n}{2}}}$. Moreover Proposition~\ref{G/Ninfinite} supports these calculations.
Using the first formula, we have
 \begin{align*}
 \log_2 |G_n|&=\log_2|N_{1,n}| + \floor*{\frac{n+1}{2}}.
 \end{align*} 
 By Lemma~\ref{lem:N1}, $\log_2|N_{1,n}|=2\log_2 | N_{3,n-1}|$ and using the second formula, we find
  \begin{align*}
\log_2 | N_{3,n-1}|&=\log_2|G_{n-1}| - \floor*{\frac{n-1}{2}}.
 \end{align*} 
 Combining these equations, we conjecture that

\begin{conj}\label{conj:G}  For any $n\geq 1$, we have 
 \begin{align*}
 \log_2 | G_{n}|&=2\log_2|G_{n-1}| - \floor*{\frac{n-3}{2}}.
 \end{align*} 
\end{conj}

 For small values of $n$, we can list the order of $G_n$ as follows.
 
 \begin{equation}\label{table:orderGn}
\begin{array}{ ||c|c|c|c|c|c|c|c|c|c|c|c|c||| } 
 \hline 
 n & 1 &2 & 3&4 &5 & 6&7& 8 &9 & 10& 11&12  \\ 
 \hline \hline
\log_2|G_n| & 1 & 3 & 6 & 12 & 23& 45 &88 & 174& 345 & 687 & 1370 & 2736 \\ 
 \hline
\end{array}
\end{equation}
\end{remark}

\begin{corollary}
 Assume the Conjecture~\ref{conj:G} holds. For every $n\geq 1$, we have $\log_2|G_n|=e_n$ where
 \[ e_n=2^{n}-1-\sum_{m=0}^{n-1}2^{n-1-m}\floor*{\frac{m}{2}}. \] 
\end{corollary}
\begin{proof}
One can check that the formula $e_n=2e_{n-1} - \floor*{\frac{n-3}{2}}$ holds. 
\end{proof}

\begin{corollary} 
Assume the Conjecture~\ref{conj:G} holds. The Hausdorff dimension of $G$ is $\frac{2}{3}$.
\end{corollary}
\begin{proof}
The group given in \cite[Proposition 2.3.1]{Pinkpolyn} for $r=2$ and the order of the group $G$ (conjecturally) are both $2^{e_n}$. In \cite[Theorem 2.3.2]{Pinkpolyn}, the Hausdorff dimension of such a group is calculated as $1-\frac{1}{2^r-1}$. We evaluate this for the value $r=2$. 
\end{proof}

\section{A Conjectural Characterization of $G$}

Fix $n$ and let $j=\floor*{\frac{n+1}{2}}$. Theorem~\ref{unity} implies that we have a homomorphism 
 \begin{equation}\label{eq:cyclotomic}
     \chi_n: G^{\text{arith}}_{n}=\Gal(K_n/k(t)) \to (\ZZ/2^j\ZZ)^*
 \end{equation}
given by the following property: for $\sigma \in G_n$, $\sigma(\zeta)=\zeta^{\chi(\sigma)}$ where $\zeta$ is a primitive $2^i$th root of unity. 

We observe that $G_{n}$ is a subgroup of the kernel of $\chi_n$.  
Let $w$ be an automorphism in $\Gal(K_{\infty}/K)$, hence it acts on the $2^n$th roots of unity for any $n\geq 1$. Remember that earlier we identified $G^{\text{arith}}$ with its image in $W$. Hence we denote the image of $w$ under the embedding $\Gal(K_{\infty}/K) \into W$ also by $w$.

In Proposition~\ref{kernelzeta}, we  describe the elements in the kernel of $\chi_n$ (which is naturally a subgroup of $\Gal(K_{\infty}/K)$) as elements $W_n$. 
We first introduce a new function.
 \begin{definition}
 Let $w \in W$ and let $\beta$ be a vertex of $T$. Define
 \begin{equation}
  \\Par(w,\beta)=\begin{cases} 0 & \text{ if }  w(\beta_1)=w(\beta)_1 \\
    1 & \text{ otherwise } 
    \end{cases}
 \end{equation}
 \end{definition}

\begin{proposition}\label{kernelzeta}
Let $w \in \ga$ such that $w\restr{T_{n-1}}=\id$. Then 
$ w(\zeta_{2^j})=\zeta_{2^j} $
if and only if 
\begin{equation}\label{eq:zeta}
\sum_{s_i \in \{1,2\}} \Par(w, t_{2s_11s_21\ldots s_{j-2}1s_{j-1}}) \equiv \sum_{s_i \in \{1,2\}} \Par(w, t_{1s_11s_21\ldots s_{j-2}1s_{j-1}}) \pmod{2}.
\end{equation}
\end{proposition}

Before we prove the proposition, we first illustrate the nodes given in ~\eqref{eq:zeta} for $n=5$ and $j=3$ as follows:
\[
		\begin{forest}
			my tree
			[$t$, fill, circle, scale=0.1, name=t
			[$t_2$, fill,circle, scale=0.1, name=t2
			[$t_{22}$ , fill, circle, scale=0.1, draw, name=t22
			[$t_{222}$, fill, circle, scale=0.1, draw
			[$t_{v2}$, fill, circle, scale=0.1, draw
			[$t_{v22}$, fill, circle, scale=0.1, draw]
			[$t_{v21}$, fill, circle, scale=0.1, draw]
			]
			[$t_{v1}$, fill, circle, scale=0.1, draw
			[$t_{v12}$, fill, circle, scale=0.1, draw]
			[$t_{v11}$, fill, circle, scale=0.1, draw]
			]]
			[\color{blue}$t_{221}$, fill, circle, scale=0.1, draw
			[$t_{s2}$, fill=red, circle, scale=0.2, draw
			[$t_{s22}$, fill, circle, scale=0.1, draw, name=s22]
			[\color{purple}$t_{s21}$, fill, circle, scale=0.1, draw, name=s21]
			]
			[\color{red}$t_{s1}$, fill=red, circle, scale=0.2, draw
			[$t_{s12}$, fill, circle, scale=0.1, draw]
			[\color{purple}$t_{s11}$, fill, circle, scale=0.1, draw]
			]
			]
			]
			[$t_{21}$,  fill, circle, scale=0.1, draw, name=t21
			[$t_{212}$, fill, circle, scale=0.1, draw, name=t212
			[$t_{n2}$, fill, circle, scale=0.1, draw, name=tn2
			[$t_{n22}$, fill, circle, scale=0.1, draw, name=n22]
			[$t_{n21}$, fill, circle, scale=0.1, draw, name=n21]
			]
			[$t_{n1}$ , fill, circle, scale=0.1, draw, name=tn1
			[$t_{n12}$, fill, circle, scale=0.1, draw]
			[$t_{n11}$, fill, circle, scale=0.1, draw]
			]
			]
			[\color{blue}$t_{211}$ , fill, circle, scale=0.1, draw, name=t211
			[$t_{m2}$, fill=red, circle, scale=0.2, draw
			[$t_{m22}$, fill, circle, scale=0.1, draw, name=m22]
			[\color{purple}$t_{m21}$, fill, circle, scale=0.1, draw, name=m21]
			]
			[\color{red}$t_{m1}$, fill=red, circle, scale=0.2, draw
			[$t_{m12}$, fill, circle, scale=0.1, draw]
			[\color{purple}$t_{m11}$, fill, circle, scale=0.1, draw]
			]
			]]]
			[$t_1$, fill, circle, scale=0.1, name=t1
			[$t_{12}$,  fill, circle, scale=0.1, draw, name=t12
			[$t_{122}$, fill, circle, scale=0.1, draw, name=t122
			[$t_{\ell2}$ , fill, circle, scale=0.1, draw
			[$t_{\ell22}$, fill, circle, scale=0.1, draw]
			[$t_{\ell21}$, fill, circle, scale=0.1, draw]
			]
			[$t_{\ell1}$,  fill, circle, scale=0.1, draw
			[$t_{\ell12}$, fill, circle, scale=0.1, draw]
			[$t_{\ell11}$, fill, circle, scale=0.1, draw]
			]]
			[\color{blue}$t_{121}$, fill, circle, scale=0.1, draw, name=t121
			[$t_{k2}$ , fill=red, circle, scale=0.2, draw
			[$t_{k22}$, fill, circle, scale=0.1, draw]
			[\color{purple}$t_{k21}$, fill, circle, scale=0.1, draw]
			]
			[\color{red}$t_{k1}$ , fill=red, circle, scale=0.2, draw
			[$t_{k12}$, fill, circle, scale=0.1, draw]
			[\color{purple}$t_{k11}$, fill, circle, scale=0.1, draw]
			]
			]
			]
			[$t_{11}$, fill , circle, scale=0.1, draw, name=t11
			[$t_{112}$, fill, circle, scale=0.1, draw, name=t112
			[$t_{j2}$ , fill, circle, scale=0.1, draw
			[$t_{j22}$, fill, circle, scale=0.1, draw,name=buraya]
			[$t_{j21}$, fill, circle, scale=0.1, draw, name=burdan]
			]
			[$t_{j1}$ , fill , circle, scale=0.1, draw
			[$t_{j12}$ , fill, circle, scale=0.1, draw ]
			[$t_{j11}$ , fill, circle, scale=0.1, draw]
			]
			]
			[\color{blue}$t_{111}$, fill , circle, scale=0.1, draw, name=t111
			[$t_{i2}$, fill=red, circle, scale=0.2, draw
			[$t_{i22}$, fill, circle, scale=0.1, draw, name=where in]
			[\color{purple}$t_{i21}$, fill, circle, scale=0.1, draw, name=where out]
			]
			[\color{red}$t_{i1}$, fill=red , circle, scale=0.2, draw
			[$t_{i12}$, fill, circle, scale=0.1, draw]
			[\color{purple} $t_{i11}$, fill, circle, scale=0.1, draw]  		]
			]]]
			]	
	\end{forest} \]
	\begin{example}
	The automorphism $w_1$ constructed in the proof of Theorem~\ref{main} is identity on the first four levels of the tree. We check that 
	\[ \sum_{s_i \in \{1,2\}}\Par(w_1, t_{2s_11s_2})=0 \text{ and } \sum_{s_i \in \{1,2\}} \Par(w_1, t_{1s_11s_2})=1 \]
	By Proposition~\ref{kernelzeta}, $w_1$ is not in $G_n$.
	\end{example}
\begin{proof}
Applying Lemma~\ref{lemma4e} inductively and using the fact that $(-1)=\frac{t_1-1}{t_2-1}$ where $f(t_1)=f(t_2)=t$, we find 

 \[ \zeta_{2^j}=\prod_{s_i \in \{1,2\}}\frac{ {(t_{1s_11s_2\ldots s_{j-1}1}}-1)}{ {(t_{2s_11s_2\ldots s_{j-1}1}-1)}} u \quad \text{or} \quad  \zeta_{2^j}=\prod_{s_i \in \{1,2\}} \frac{{(t_{2s_11s_2\ldots s_{j-1}1}}-1)}{{(t_{1s_11s_2\ldots s_{j-1}1}-1)}} u\]
 where $u$ is a product of expressions involving the vertices of level at most $n-1$. For example, if $n=3$, then $j=2$ and we have 
 \[ \zeta_4=\frac{\prod_s(t_{2s1}-1)}{\prod_s (t_{1s1}-1)} \frac{(t_{21}-1)}{(t_{11}-1)}. \] Here we take $u=\frac{(t_{21}-1)}{(t_{11}-1)}$.

Let $\beta$ be a vertex of $T$ and let $w\in W$. Since $(\beta_1-1)=-(\beta_2-1)$, we have $$w(\beta_i-1)=(-1)^{\Par(w,\beta)}(w(\beta)_i-1)$$ for $i=1,2$. Using this property, we calculate that
 \begin{equation}
  w(\zeta_{2^j})=(-1)^{(\sum \\Par(w, t_{2s_11s_2\ldots s_{j-1}}) -\sum \Par(w, t_{1s_11s_2\ldots s_{j-1}}))}  \frac{\prod {
  (w(t_{1s_11s_2\ldots s_{j-1}})_1-1)}}{\prod {(w(t_{2s_11s_2\ldots s_{j-1}})_1-1)}} w(u).             
 \end{equation}
 Since $u$ is an expression of vertices of level at most $n-1$, $w(u)=u$. If $n$ is even, then $n=2j$ and if $n$ is odd, $n=2j-1$. Hence in both cases, $w(t_{2s_11s_2\ldots 1s_{j-1}})=t_{2s_11s_2\ldots 1s_{j-1}}$ and $w(t_{2s_11s_2\ldots 1s_{j-1}})=t_{1s_11s_2\ldots 1s_{j-1}}$ and the statement follows.
 
\end{proof}


 \begin{remark}
 When $n$ is even, $j=\frac{n}{2}$ and $t_{1s_11s_2\ldots 1s_{j-2}1s_{j-1}}$ and $t_{2s_11s_21\ldots s_{j-2}1s_{j-1}}$ are vertices of level $2(j-1)=n-2$. Hence the condition given in Proposition~\ref{kernelzeta} is automatically satisfied. 
 \end{remark}

\begin{example}
Let $n=3$, then $j=2$. Let $w \in G^{\text{arith}}_3$ such that $w\restr {T_{2}}=\id$. The calculation above implies that $w(\zeta_4)=\zeta_4$ if and only if $\Par(w,t_{11})+\Par(w,t_{12})=\Par(w,t_{21})+\Par(w,t_{22})$. This equation implies that $\sgn_3(w)=1$, hence $w$ is in $G_3$ as we proved in Proposition~\ref{G-E isom. n=3}.
\end{example}

Using Proposition~\ref{kernelzeta}, we show that the quantity described in Conjecture~\ref{conj:G} is the maximum value that the group $G_n$ can attain. In particular, we prove that $G_n$ is completely determined by the relation in \ref{eq:zeta}.

  \begin{theorem}\label{thm:cyc}
Assume Conjecture~\ref{conj:G} holds. Then 
\begin{enumerate}
\item The quotient group $\ga/G$ is isomorphic to a finite index subgroup of $\ZZ_2^*$. 
\item The homomorphism $\Gal(\bar{k}/k) \to \ZZ_2^*$ is given by the cyclotomic character.
\end{enumerate}
 In particular the field $F$ in \eqref{eq:exact} is $k(\{\zeta_{2^j} : j\geq 1\})$.
\end{theorem}

\begin{proof}    
 For any $n\geq 1$, $G_n$ is a subgroup of $\ker(\chi_n)$. 
Let $U_n$ denote the kernel of the natural projection $G_n \to G_{n-1}$.  We know by the self-similarity of $G_n$ that $U_n$ is a subgroup of $U_{n-1}\times U_{n-1}$. By Proposition~\ref{kernelzeta}, it satisfies the relation \ref{eq:zeta}. We assume that for any $n\geq 2$, $U_n$ consists of the elements of $U_{n-1}\times U_{n-1}$ satisfying ~\ref{eq:zeta}. We use this to find an upper bound for the order of $G_n$.

With this assumption, we first claim that for any $n\geq 2$, 
\[
 | (U_{n-1}\times U_{n-1}) :U_n |=\begin{cases} 1 & \text{ if } n \text{ is even } \\ 2 & \text{ if } n \text{ is odd } \end{cases}
\]
Note that ~\ref{eq:zeta} defines a homomorphism from $U_{n-1} \times U_{n-1} \to \ZZ/2\ZZ$ since $U_n$ is abelian for any $n\geq 2$ and by our assumption on $U_n$, the kernel of this map is exactly $U_n$. Hence the aforementioned index is either one or two. Next, we construct a sequence of elements $v_n \in W_n$ as follows: 

Let $v_1=\sigma$ and for any $n\geq 2$, we define
\[ v_n=\begin{cases} (v_{n-1},v_{n-1}) & \text{ if } n \text{ is odd } \\ 
           (v_{n-1},\id) & \text{ if } n \text{ is even } \end{cases}
\]
The automorphisms $v_n$ satisfy the relation given in ~\ref{eq:zeta}. Hence $v_n$ is in $U_n$ for any $n\geq 1$. Also $(v_{n-1},\id)$ is an element of $U_{n-1}\times U_{n-1}$ but not in $U_n$ when $n$ is odd. This proves that the index in the claim is two when $n$ is odd.

Let $e_n=\log_2 |G_n|$ and let $s_n=\log_2|U_n|$. We have $e_n=e_{n-1}+s_n$. By our assumption, we have
\[ s_n-2s_{n-1}=\begin{cases} 0 & \text{ if } n \text{ is even } \\
-1 & \text{ otherwise. }
\end{cases}
\]
We calculate first that $s_n-2s_{n-1}=e_n-e_{n-1}-2e_{n-1}+2e_{n-2}$. By induction, assume $e_{m}=2e_{m-1} - \floor*{\frac{m-3}{2}}$ for all $m\leq n-1$. Then  
\[ s_n-2s_{n-1}=e_n-2e_{n-1} -(e_{n-1}-2e_{n-2})=e_n-2e_{n-1}+ \floor*{\frac{n-4}{2}} \]

We obtain 
\[ e_n-2e_{n-1} =\begin{cases} -( \frac{n-4}{2} ) & \text{ if } n \text{ is even } \\
-(\frac{n-3}{2}) & \text{ otherwise. }
\end{cases}
\]
Hence  $e_n=2e_{n-1} - \floor*{\frac{n-3}{2}}$. This shows that the order of $G_n$ is at most the quantity described in Conjecture~\ref{conj:G} and it is achieved when $G_n$ is exactly the set of automorphisms in $G^{\text{arith}}_n$ fixing the roots of unity. Hence the statement follows.
\end{proof}

\bibliographystyle{plain}
\bibliography{refmon.bib}
\end{document}